\documentclass{article}

\advance\topmargin -1.05cm
\advance\textheight 2.1cm
\advance\evensidemargin -1cm
\advance\oddsidemargin -1cm
\advance\textwidth 2cm

\usepackage{amssymb}
\usepackage{amsfonts}
\usepackage{amsmath}
\usepackage{multirow}
\usepackage[amsmath,amsthm,thmmarks]{ntheorem}
\usepackage[round,comma,authoryear]{natbib}
\usepackage{graphicx}
\usepackage{color}

\usepackage{floatrow}
\newcommand{\down}[1]{\raisebox{-0.5ex}{#1}}
\newcommand{\ddown}[1]{\raisebox{-1ex}{#1}}

\newtheorem{theorem}{Theorem}

\newtheorem{proposition}[theorem]{Proposition}
\newtheorem{assumption}{Assumption}

\newtheorem{definition}[theorem]{Definition}
\newtheorem*{remark}{Remark}

\newcommand{\sign}{\operatorname{sgn}}

\newcommand{\N}{\mathbb{N}}
\newcommand{\R}{\mathbb{R}}
\newcommand{\Rquer}{\overline{\R}}
\newcommand{\ind}{\boldsymbol{1}}
\newcommand{\eps}{\varepsilon}

\begin{document}

\title{Confidence Sets Based on Thresholding Estimators in High-Dimensional
  Gaussian Regression Models} 

\author{Ulrike Schneider\thanks{Institute for Mathematical Methods in
    Economics, Vienna University of Technology, 
    Argentinierstra{\ss}e 8/E105-2, A-1040 Vienna,
    ulrike.schneider@tuwien.ac.at}\\ 
    Vienna University of Technology}

\date{}

\maketitle

\begin{abstract}
We study confidence intervals based on hard-thresholding,
soft-thresholding, and adaptive soft-thresholding in a linear
regression model where the number of regressors $k$ may depend on and
diverge with sample size $n$. In addition to the case of known error
variance, we define and study versions of the estimators when the
error variance is unknown. In the known-variance case, we provide an
exact analysis of the coverage properties of such intervals in finite
samples. We show that these intervals are always larger than the
standard interval based on the least-squares
estimator. Asymptotically, the intervals based on the thresholding
estimators are larger even by an order of magnitude when the
estimators are tuned to perform consistent variable selection. For the
unknown-variance case, we provide non-trivial lower bounds and a small
numerical study for the coverage probabilities in finite samples. We
also conduct an asymptotic analysis where the results from the
known-variance case can be shown to carry over asymptotically if the
number of degrees of freedom $n-k$ tends to infinity fast enough in
relation to the thresholding parameter.
\end{abstract}

\section{Introduction}
\label{intro}

We study confidence sets based on thresholding estimators such as
hard-thresholding, soft-thresholding, and adaptive soft-thresholding
in a Gaussian linear regression model when the number of regressors
can be large. When the regressors are orthogonal, the estimators we
consider can be viewed as penalized least-squares estimators, with
soft-thresholding then coinciding with the Lasso (introduced by
\citealp*{frafri93}, \citealp*{allruz94}, and \citealp*{tib96}) and
adaptive soft-thresholding coinciding with the adaptive Lasso
\citep[introduced by][]{zou06}. Thresholding estimators have of course
been discussed earlier in the context of model selection
\citep[see][]{baupoehac88} and in the context of wavelets \cite[see,
  e.g.][]{donetal95}.

These types of estimators also widely gained importance in the context
of econometric models: \cite{belche11} provide a discussion of using
Lasso-type estimators in econometrics, including applications to
earning regressions, instrument selection, and cross-country growth
regressions. \cite{cankni13} introduce Bridge estimators (of which the
Lasso is a special case) to the to the framework of unit root tests,
whereas \cite{canzha14} investigate the elastic net estimator (merging
the Lasso with ridge regression) in the context of generalized method
of moments estimators. For using thresholding estimators in connection
with econometric models, see e.g.\ \cite{baing08} who use hard- and
soft-thresholding for factor forecasting.

When using and applying these kinds of estimators to econometric (and
other) models, it is of course of importance to know about the
statistical performance of the estimator in question, in particular of
how to perform valid inference (see also \cite{beretal13} for a
treatment within a different context). For this, knowledge of the
distributional properties of the particular estimator is
needed. Contributions concerning such properties of thresholding and
penalized least-squares estimators include the following:
\cite{knifu00} derive the asymptotic distribution of the Lasso and
related estimators when they are tuned to act as a conservative
variable selection procedure, while the asymptotic distribution of the
Lasso and the adaptive Lasso estimator when tuned as a consistent
variable selection procedure is considered in \cite{zou06}. The
asymptotic distribution of the so-called smoothly clipped absolute
deviation (SCAD) estimator is derived in \cite{fanli01} and
\cite{fanpen04} also under consistent tuning. Following this work,
many papers have been published that study the asymptotic distribution
of various penalized least-squares estimators when those are tuned to
act as a consistent variable selection procedure, see the introduction
in \cite{poesch09} for a partial list. With the exception of
\cite{knifu00}, all these papers consider a so-called fixed-parameter
framework for the asymptotic results. But as pointed out in
\cite{leepoe05}, such a framework may be highly misleading in the
context of variable selection procedures and penalized least-squares
estimators. To this end, \cite{poelee09} and \cite{poesch09} carry out
a detailed study of the finite-sample as well as large-sample
distribution of various penalized least-squares estimators, adopting a
moving-parameter framework for the asymptotic results. While these two
papers are set in the framework of an orthogonal linear regression
model with a fixed number of parameters and known error variance,
\cite{poesch11} investigate finite-sample and large-sample
distributions for thresholding estimators also for non-orthogonal
regressors and a potentially diverging number of parameters.

Given all these distributional results in finite samples as well as
asymptotically within a moving parameter framework, a natural question
is of course what these results imply for confidence sets, as this is
an important issue for statistical inference. To address this
question, \cite{poesch10} consider confidence intervals based on
hard-thresholding, Lasso, and adaptive Lasso within an orthogonal
linear regression model with normal errors. \emph{In the present
  paper, we extend these results relaxing the condition of
  orthogonality as well as allowing for a high-dimensional framework
  where the number of parameters $k$ may depend on and diverge with
  sample size $n$. The estimators we consider are hard-, soft-, and
  adaptive soft-thresholding acting componentwise.} In addition to the
case of known error variance $\sigma^2$, we define versions of the
estimators when the error variance is unknown. We also make use of
some distributional results derived in \cite{poesch11}.

Our main contributions and findings are as follows. In the case of
known error variance, we derive explicit expressions for the minimal
coverage probabilities of fixed (non-random) width confidence
intervals in finite samples based on the estimators in question. We
show that symmetric intervals are the shortest. The interval based on
soft-thresholding is smaller than the one based on adaptive
soft-thresholding which in turn is smaller than the one based on
hard-thresholding. Compared to the standard interval based on the
least-squares estimator, the intervals based on the thresholding
estimators are all larger in finite samples. Asymptotically, two
different pictures arise: when the estimators are tuned to perform
conservative model selection, the lengths of the intervals are all of
the same order which essentially is $n^{-1/2}$. If the estimators are
tuned to carry out consistent selection, however, it turns out the
lengths of the intervals based on the thresholding estimators are
larger by an order of magnitude compared to the standard one based on
least-squares estimation.

In the case of unknown error variance, we consider symmetric intervals
of random width (where the randomness is introduced by scaling the
length with an estimate of the error standard deviation). We provide
lower and upper bounds for the minimal coverage probabilities of the
intervals based on the thresholding estimators under
consideration. When comparing the lengths of the intervals based on
the thresholding estimators to the standard one based on least-squares
estimation, asymptotically we arrive at the same conclusions as in the
known-variance case. We also find that the effect of having to
estimate the error variance disappears asymptotically only the number
of degrees of freedom $n-k$ diverges fast enough in relation to the
tuning (thresholding) parameter of the estimators.

The paper is organized as follows. We introduce the model and define
the estimators and some notation in Section~\ref{model}.
Sections~\ref{fsdist} and \ref{lsdist} derive auxiliary results for
the finite- and large sample distributions of the thresholding
estimators which are needed later for the study of confidence
intervals in the main Section~\ref{confs}. Section~\ref{summ} contains
a short summary, and a table with an overview of assumptions and
results is given in Appendix~\ref{table}. Finally, proofs are
relegated to Appendix~\ref{proofs}.

\section{The Model and the Estimators}
\label{model}

\setcounter{assumption}{12} 
\begin{assumption}

The model we consider is the linear regression model 
$$
y=X\theta +u,
$$ 
where $y$ is the $n \times 1$ data vector, $X$ is a non-stochastic $n
\times k$ matrix of rank $k \geq 1$ containing the regressors, and $u
\sim N(0,\sigma^2 I_n)$. 

\end{assumption}
We allow $k$, the number of columns of $X$, as well as the entries of
$y$, $X$, and $u$ to depend on sample size $n$, although we almost
always suppress this dependence on $n$ in the notation. Note that this
framework allows for high-dimensional regression models, where the
number of regressors $k$ may diverge, as well as for the more
classical situation where $k$ remains fixed. Let
\begin{eqnarray*}
\hat\theta_{LS} & = & \left(X'X\right)^{-1}X'y \\
\hat\sigma^2 & = & (y-X\hat\theta_{LS})'(y-X\hat\theta_{LS})/(n-k)
\end{eqnarray*}
denote the least-squares estimator for $\theta$ and the associated
estimator for $\sigma^2$, the latter being defined only if
$n>k$. Furthermore, we denote by $\xi_{i,n}$ the non-negative square
root of $((X'X/n)^{-1})_{ii}$, the $i$-th diagonal element of
$(X'X/n)^{-1}$. Note that in the textbook case where $k$ is fixed and
$X'X/n$ is assumed to converge to a finite positive definite matrix,
$\xi_{i,n}$ asymptotically settles at a finite value greater than
zero.

\begin{definition}
\label{ests}
The hard-thresholding estimator $\tilde\theta_H =
(\tilde\theta_{H,1},\dots,\tilde\theta_{H,k})'$ is defined via its
components as follows
$$
\tilde\theta_{H,i} = \tilde\theta_{H,i}(\eta_{i,n}) =
\hat\theta_{LS,i} \, \ind\left(\left|\hat\theta_{LS,i}\right| > 
\hat\sigma\xi_{i,n}\eta_{i,n}\right),
$$ 
where the tuning or thresholding parameters $\eta_{i,n}$ are positive
real numbers that may change over components and $\hat\theta_{LS,i}$
denotes the $i$-th component of the least-squares estimator. We also
consider its infeasible counterpart $\hat\theta_H =
(\hat\theta_{H,i},\dots,\hat\theta_{H,i})$ given by
$$
\hat\theta_{H,i} = \hat\theta_{H,i}(\eta_{i,n}) = 
\hat\theta_{LS,i} \, \ind\left(\left|\hat\theta_{LS,i}\right| > 
\sigma\xi_{i,n}\eta_{i,n}\right),
$$ 
assuming knowledge of the error variance $\sigma^2$. The
soft-thresholding estimator\footnote{If the regressor matrix $X$
  contains orthogonal columns, the soft-thresholding estimator
  coincides with the Lasso, the Dantzig and the Elastic Net
  estimator.} $\tilde\theta_S$ and its infeasible counterpart
$\hat\theta_S$ are given by the components
\begin{eqnarray*}
\tilde\theta_{S,i} & = & \tilde\theta_{S,i}(\eta_{i,n}) \; = \;
\sign(\hat\theta_{LS,i})\left(\left|\hat\theta_{LS,i}\right| - 
\hat\sigma\xi_{i,n}\eta_{i,n}\right)_{+} \\
& = & \left\{
\begin{array}{cc}
0 & \text{ if } |\hat\theta_{LS,i}| \leq \hat\sigma\xi_{i,n}\eta_{i,n} \\ 
\hat\theta_{LS,i} - \sign(\hat\theta_{LS,i})\hat\sigma\xi_{i,n}\eta_{i,n} &
\text{ if } |\hat\theta_{LS,i}| > \hat\sigma \xi_{i,n}\eta_{i,n} 
\end{array}
\right.
\end{eqnarray*}
and
\begin{eqnarray*}
\hat\theta_{S,i} & = & \hat\theta_{S,i}(\eta_{i,n}) \; = \;
\sign(\hat\theta_{LS,i})\left(\left|\hat\theta_{LS,i}\right| - 
\sigma\xi_{i,n}\eta_{i,n}\right)_{+} \\
& = & \left\{
\begin{array}{cc}
0 & \text{ if } |\hat\theta_{LS,i}| \leq \sigma\xi_{i,n}\eta_{i,n} \\ 
\hat\theta_{LS,i}  - \sign(\hat\theta_{LS,i})\sigma\xi_{i,n}\eta_{i,n} &
\text{ if } |\hat\theta_{LS,i}| > \sigma\xi_{i,n}\eta_{i,n}, 
\end{array}
\right.
\end{eqnarray*}
where $(\cdot)_{+} = \max(\cdot,0)$. Finally, the adaptive
soft-thresholding estimator\footnote{If the regressor matrix $X$
  contains orthogonal columns, the adaptive soft-thresholding estimator
  coincides with the adaptive Lasso estimator.} $\tilde\theta_{AS}$
and its infeasible counterpart $\hat\theta_{AS}$ are defined
componentwise via
\begin{eqnarray*}
\tilde\theta_{AS,i} & = & \tilde\theta_{AS,i}(\eta_{i,n}) \; = \; 
\hat\theta_{LS,i} \left(1 - 
\hat\sigma^{2}\xi_{i,n}^2\eta_{i,n}^2/\hat\theta_{LS,i}^2\right)_{+} \\ 
& = & \left\{
\begin{array}{cc}
0 & \text{ if } |\hat\theta_{LS,i}| \leq \hat\sigma\xi_{i,n}\eta_{i,n} \\ 
\hat\theta_{LS,i} - \hat\sigma^2\xi _{i,n}^2\eta_{i,n}^2/\hat\theta_{LS,i} & 
\text{ if } |\hat\theta_{LS,i}| > \hat\sigma\xi_{i,n}\eta_{i,n}
\end{array}
\right.
\end{eqnarray*}
and 
\begin{eqnarray*}
\hat\theta_{AS,i} & = & \hat\theta_{AS,i}(\eta_{i,n}) \; = \; 
\hat\theta_{LS,i} \left(1 - 
\sigma^2\xi_{i,n}^2\eta_{i,n}^2/\hat\theta_{LS,i}^2\right)_{+} \\ 
& = & \left\{
\begin{array}{cc}
0 & \text{ if } |\hat\theta_{LS,i}| \leq \sigma\xi_{i,n}\eta_{i,n} \\ 
\hat\theta_{LS,i} - \sigma^2\xi _{i,n}^2\eta_{i,n}^2/\hat\theta_{LS,i} & 
\text{ if } |\hat\theta_{LS,i}| > \sigma\xi_{i,n}\eta_{i,n}.
\end{array}
\right.
\end{eqnarray*}

\end{definition}
Note that $\tilde\theta_H$, $\tilde\theta_S$, and $\tilde\theta_{AS}$
as well as their infeasible counterparts are equivariant under scaling
of the columns of $(y:X)$ by non-zero column-specific scale factors.
We have chosen to let the thresholds $\hat\sigma\xi_{i,n}\eta_{i,n}$
($\sigma\xi_{i,n}\eta_{i,n}$, respectively) depend explicitly on
$\hat\sigma$ ($\sigma$, respectively) and $\xi_{i,n}$ in order to give
$\eta_{i,n}$ an interpretation independent of the values of $\sigma $
and $X$. Often $\eta_{i,n}$ will be chosen independently of $i$, i.e.,
$\eta_{i,n}=\eta_n$ where $\eta_n$ is a positive real number. Clearly,
for the feasible versions we always need to assume $n > k$, whereas
for the infeasible versions $n \geq k$ suffices.

\medskip

Aside from requiring that the regressor matrix $X$ has full column
rank, we essentially have no assumptions on the regressor matrix $X$
except that for simplicity for all asymptotic considerations, we
assume the following.

\setcounter{assumption}{0} 
\begin{assumption}
\label{xi}
Let $\xi_{i,n}^2/n = ((X'X)^{-1})_{i,i}$ satisfy
$$
\sup_n\xi_{i,n}^2/n < \infty  
$$
for every fixed $i \geq 1$ satisfying $i \leq k(n)$ for large enough
$n$.

\end{assumption}
Note that Assumption~\ref{xi} is not really restrictive in the sense
that the case excluded by it implies unboundedness of $\xi_{i,n}^2/n$,
which in particular would entail inconsistency of the least-squares
estimator\footnote{In fact, if $k$ is fixed and for each $n$, the
  regressor matrix $X$ changes only by appending an additional row,
  unboundedness of $\xi_{i,n}^2/n$ is impossible since then the
  diagonal elements of $(X'X)^{-1}$ are monotonically decreasing.}.

\medskip

We now turn to some definitions in terms of the asymptotic regimes we
consider. Clearly, all three estimators exhibit positive probability
of being set equal to 0 and in that sense they perform variable
selection. Asymptotically, we will distinguish two different cases
for this. Let $\breve\theta_i$ denote any of the thresholding
estimators introduced above.

\begin{definition}

The case of consistent variable selection occurs when
$$
\lim_{n \to \infty} P_{n,\theta,\sigma}(\breve\theta_i=0) = 1 
\;\; \text{ whenever } \theta_i = 0,
$$ 
in which we shall refer to $\breve\theta_i$ as being
consistently tuned. The other case is the case of
conservative variable selection where
$$
\limsup_{n \to \infty} P_{n,\theta,\sigma}(\breve\theta_i=0) < 1 
\;\; \text{ whenever } \theta_i = 0,
$$ in which we shall call $\breve\theta_i$ conservatively
  tuned. 

\end{definition}
Propositions~4 and 10 in \cite{poesch11} show that $\breve\theta_i$ is
consistently tuned when $n^{1/2}\eta_{i,n} \to \infty$, and
conservatively tuned when $n^{1/2}\eta_{i,n} \to e_i$ with $0 \leq e_i
< \infty$, including the case $e_i = 0$ in which $\breve\theta_i$ can
be shown to be asymptotically equivalent to $\hat\theta_{LS,i}$ (see
Remark~17 in the above reference). Moreover, Theorem~16 in the same
reference shows that $\breve\theta_i$ is consistent for $\theta_i$ in
terms of parameter estimation whenever $\xi_{i,n}\eta_{i,n} \to 0$ (in
fact, the estimator is then even uniformly consistent) and this
assumption will appear as a basic condition in all asymptotic
considerations.

\medskip

We conclude this section by introducing some more notation:
$\hat\theta_i$ denotes any of the estimators $\hat\theta_{H,i}$,
$\hat\theta_{S,i}$, or $\hat\theta_{AS,i}$ and $\tilde\theta_i$ any of
the estimators $\tilde\theta_{H,i}$, $\tilde\theta_{S,i}$, or
$\tilde\theta_{AS,i}$. Let $\Rquer$ be the extended real line $\R \cup
\{-\infty,\infty\}$. Furthermore, $\Phi$ and $\phi$ are the cumulative
distribution function (cdf) and the probability density function (pdf)
of a standard normal distribution, respectively. By $T_m$ and $t_m$ we
denote the cdf and pdf of a $t$-distribution with $m \in \N$ degrees
of freedom, respectively. We use the convention $\Phi (\infty )=1$,
$\Phi(-\infty)=0$ with a similar convention for $T_m$. By $\rho_m$ we
denote the density function of $\sqrt{\chi^2_m/m}$, the square root of
a chi-squared-distributed random variable divided by its degrees of
freedom, $m$. For repeated later use, note that $T_m(x) =
\int_0^\infty \Phi(xs)\rho_m(s)\,ds$. Finally, $\delta_z$ will denote
the measure for pointmass at $z$.

\section{Auxiliary Results: Finite-Sample Distributions}
\label{fsdist}

\cite{poesch11} derive finite- and large-sample distributions of the
thresholding estimators defined in Definition~\ref{ests} for the case
of known and of unknown error variance. More concretely, the
finite-sample distributions of $\sigma^{-1} \alpha_{i,n} (\hat\theta_i
- \theta_i)$ and $\sigma^{-1} \alpha_{i,n} (\tilde\theta_i -
\theta_i)$ are derived, where $\alpha_{i,n} > 0$ is a non-random
scaling factor. When considering the large-sample distributions, the
scaling factor $\alpha_{i,n}$ is set equal to $n^{1/2}/\xi_{i,n}$ in
case of conservative tuning and equal to $(\xi_{i,n}\eta_{i,n})^{-1}$
in case of consistent tuning. These are shown to correspond to the
uniform convergence rates of the estimators.

In the present paper, to analyze the coverage properties of confidence
sets based on $\hat\theta_i$ in Section~\ref{confknown}, we will make
use of the finite-sample distributions of $\sigma^{-1} \alpha_{i,n}
(\hat\theta_i - \theta_i)$. These distributions were derived in in
Propositions~19-21 \cite{poesch11}. In the unknown-variance case, to
investigate the coverage probabilities of confidence sets based on
$\tilde\theta_i$ in Section~\ref{confunknown}, we will need knowledge
of the distributions of $\hat\sigma^{-1} \alpha_{i,n} (\tilde\theta_i
- \theta_i)$ with random scaling by $\hat\sigma^{-1}$ rather than the
distributions of $\sigma^{-1} \alpha_{i,n} (\tilde\theta_i -
\theta_i)$ with non-random scaling which have been considered in the
above mentioned paper.

To this end, we derive the finite-sample distributions of
$\hat\sigma^{-1} \alpha_{i,n} (\tilde\theta_i - \theta_i)$ in the
following propositions and make some qualitative comparisons to the
distributions of $\sigma^{-1} \alpha_{i,n} (\tilde\theta_i -
\theta_i)$. The corresponding large-sample distributions are
considered in Section~\ref{lsdist}.

We shall suppress the dependence of the distribution function on the
scaling factor $\alpha_{i,n}$ in the notation. Moreover, note that the
distribution functions depend on the parameter $\theta$ only through
the $i$-th component $\theta_i$. We start by considering the
hard-thresholding estimator.


\begin{proposition}[Hard-thresholding in finite samples]
\label{fsdistH}
The cdf $\tilde H^i_{H,n,\theta,\sigma} := \tilde
H^i_{H,\eta_{i,n},n,\theta,\sigma}$ of $\hat\sigma^{-1}\alpha_{i,n}
(\tilde\theta_{H,i} - \theta_i)$ is given by
\begin{align}
\label{fscdfH}
\begin{split}
&\tilde H^i_{H,n,\theta,\sigma}(x) \; = \; \int_0^\infty 
\left[ \Phi\big(n^{1/2}xs/(\alpha_{i,n}\xi_{i,n})\big)
\ind\left(|xs/\alpha_{i,n} + \theta_i/\sigma| > \xi_{i,n}s\eta_{i,n}
\right) \right. \\
& + \; \Phi\big(n^{1/2} (-\theta_i/(\sigma \xi_{i,n}) + s\eta_{i,n})\big)
\ind\left(0 \leq xs/\alpha_{i,n} + \theta_i/\sigma \leq \xi_{i,n}s\eta_{i,n}
\right) \\
& \left. + \;
\Phi\big(n^{1/2} (-\theta_i/(\sigma \xi_{i,n}) - s\eta_{i,n})\big)
\ind\left(-\xi_{i,n}s\eta_{i,n} \leq xs/\alpha_{i,n} + \theta_i/\sigma < 0 
\right) \right] \rho_{n-k}(s) \, ds,
\end{split}
\end{align}
or equivalently, its measure is given by
\begin{align}
\label{fspdfH0}
\begin{split}
d&\tilde H^i_{H,n,\theta,\sigma}(x) = 
\int_0^\infty \left[\Phi\big(n^{1/2}s\eta_{i,n}\big) - 
\Phi\big(-n^{1/2}s\eta_{i,n}\big) \right] \rho_{n-k}(s)\,ds\,d\delta_0(x) \\
& + \ind\left(|x| > \alpha_{i,n}\xi_{i,n}\eta_{i,n}\right)
\int_0^\infty (n^{1/2}s/(\alpha_{i,n}\xi_{i,n})) 
\phi\big(n^{1/2}xs/(\alpha_{i,n}\xi_{i,n})\big)
\rho_{n-k}(s) \, ds dx
\end{split}
\end{align}
for $\theta_i = 0$, and by
\begin{align}
\label{fspdfHnot0}
\begin{split}
d&\tilde H^i_{H,n,\theta,\sigma}(x) = 
\; \ind\{-\sign(\theta_i)\,x \geq 0\} \, \alpha_{i,n}|\theta_i|/(\sigma x^2)
\, \rho_{n-k}\left(-\alpha_{i,n}\theta_i/(\sigma x)\right) \\
& \times \left[\Phi\big(-(n^{1/2}\theta_i/(\sigma\xi_{i,n}))(1 +
\alpha_{i,n}\xi_{i,n}\eta_{i,n}/x)\big) - 
\Phi\big(-(n^{1/2}\theta_i/(\sigma\xi_{i,n}))
(1 - \alpha_{i,n}\xi_{i,n}\eta_{i,n}/x)\big)\right] \, dx\\
& + \int_0^\infty 
(n^{1/2}s/(\alpha_{i,n}\xi_{i,n})) 
\phi\big(n^{1/2}xs/(\alpha_{i,n}\xi_{i,n})\big)
\ind\left(|xs/\alpha_{i,n} + \theta_i/\sigma| > \xi_{i,n}s\eta_{i,n}\right) 
\rho_{n-k}(s) \, ds dx
\end{split}
\end{align}
in case $\theta_i \neq 0$.
\end{proposition}
\begin{figure}[htb]
\begin{center}
\includegraphics[width=6cm]{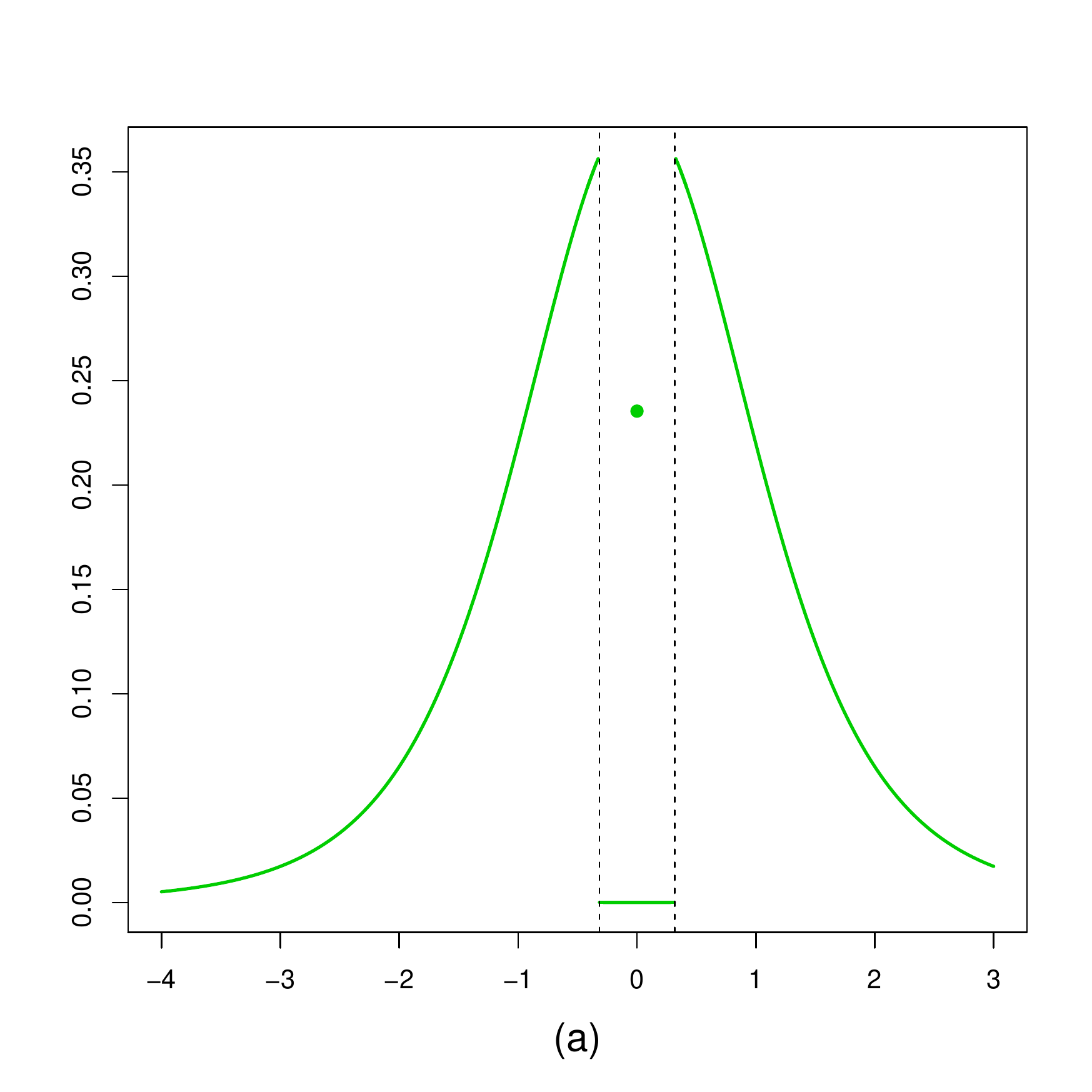}
\includegraphics[width=6cm]{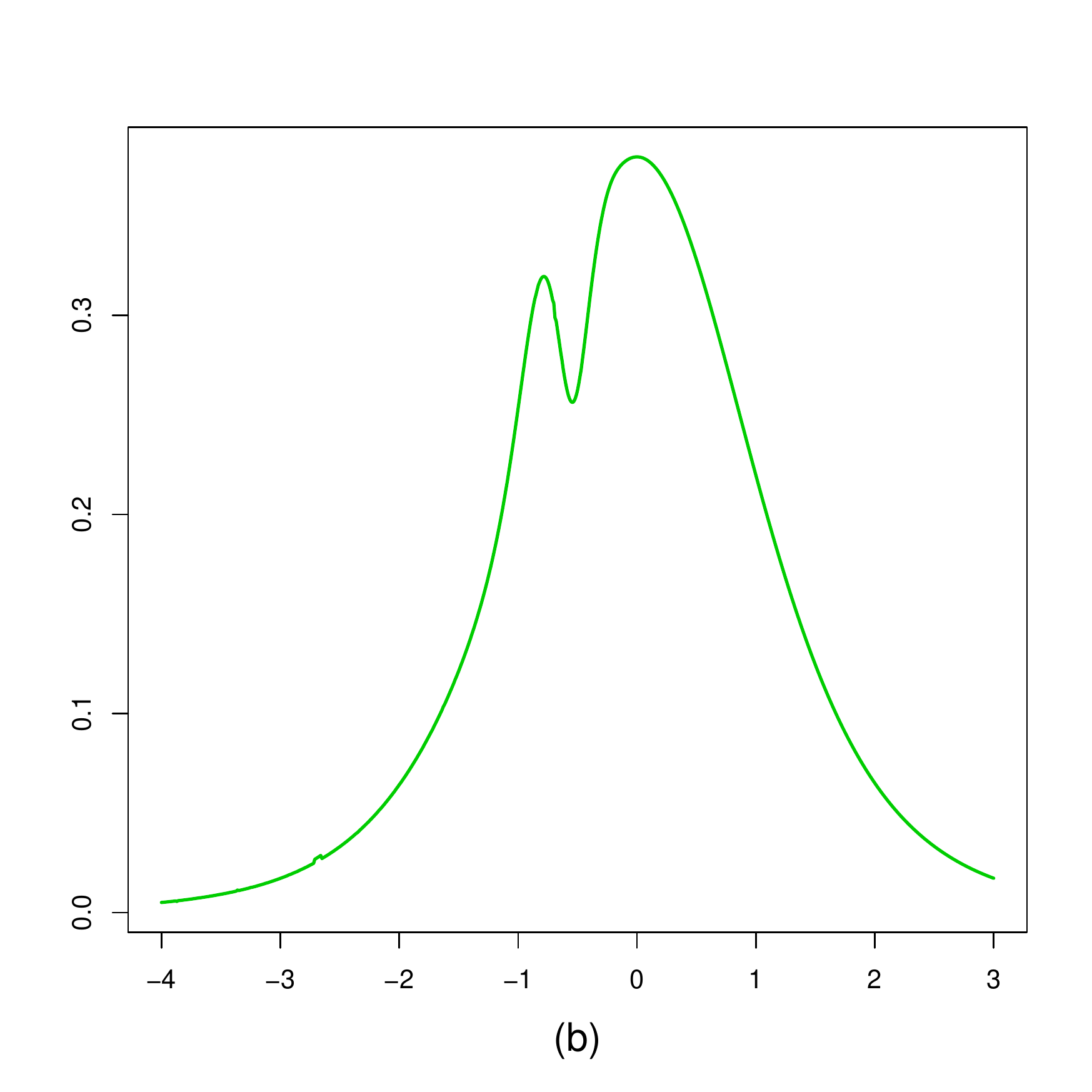}
\caption{\label{plotpdfH} Hard-thresholding: Plots of $d\tilde
  H^i_{H,n,\theta,\sigma}$ for (a) $\theta_i=0$ and (b)
  $\theta_i=0.16$. Exemplarily, in both plots we set $n=40$, $k=35$,
  $\eta_{i,n}=0.05$, $\xi_{i,n}=1$, $\sigma^2=1$,
  $\alpha_{i,n}=n^{1/2}/\xi_{i,n}$. The dot in part (a) corresponds to
  the total mass of the atomic part.}
\end{center}
\end{figure}
We now look at the soft-thresholding estimator.

\begin{proposition}[Soft-thresholding in finite samples]
\label{fsdistS}
The cdf $\tilde H^i_{S,n,\theta,\sigma} := \tilde
H^i_{S,\eta_{i,n},n,\theta,\sigma}$ of $\hat\sigma^{-1}\alpha_{i,n}
(\tilde\theta_{S,i} - \theta_i)$ is given by
\begin{align}
\label{fscdfS}
\begin{split}
&\tilde H^i_{S,n,\theta,\sigma}(x) \; = \; \int_0^\infty 
\left[ \Phi\big(n^{1/2}(xs/(\alpha_{i,n}\xi_{i,n})+s\eta_{i,n})\big)
\ind\left(xs/\alpha_{i,n} + \theta_i/\sigma \geq 0\right) \right. \\
& + \; \left.\Phi\big(n^{1/2}(xs/(\alpha_{i,n}\xi_{i,n})-s\eta_{i,n})\big)
\ind\left(xs/\alpha_{i,n} + \theta_i/\sigma < 0\right) \right]
\rho_{n-k}(s) \, ds,
\end{split}
\end{align}
or equivalently, its measure is given by
\begin{align}
\label{fspdfS0}
\begin{split}
d&\tilde H^i_{S,n,\theta,\sigma}(x) = 
\int_0^\infty \left[\Phi\big(n^{1/2}s\eta_{i,n}\big) - 
\Phi\big(-n^{1/2}s\eta_{i,n}\big) \right] \rho_{n-k}(s)\,ds\,d\delta_0(x) \\
& + \int_0^\infty (n^{1/2}s/(\alpha_{i,n}\xi_{i,n})) \left[
\phi\big(n^{1/2}(xs/(\alpha_{i,n}\xi_{i,n})+s\eta_{i,n})\big)
\ind\left(x > 0 \right)\right.\\
&+\left. \phi\big(n^{1/2}(xs/(\alpha_{i,n}\xi_{i,n})-s\eta_{i,n})\big)
\ind\left(x < 0 \right)\right] \rho_{n-k}(s) \, ds dx
\end{split}
\end{align}
for $\theta_i = 0$ and by
\begin{align}
\label{fspdfSnot0}
\begin{split}
d&\tilde H^i_{S,n,\theta,\sigma}(x) = 
\; \ind\{-\sign(\theta_i)\,x \geq 0\} \, \alpha_{i,n}|\theta_i|/(\sigma x^2)
\, \rho_{n-k}\left(-\alpha_{i,n}\theta_i/(\sigma x)\right) \\
& \times \left[\Phi\big(-(n^{1/2}\theta_i/(\sigma\xi_{i,n}))
(1 + \alpha_{i,n}\xi_{i,n}\eta_{i,n}/x)\big) - 
\Phi\big(-(n^{1/2}\theta_i/(\sigma\xi_{i,n}))
(1 - \alpha_{i,n}\xi_{i,n}\eta_{i,n}/x)\big)\right] \, dx\\
& + \int_0^\infty (n^{1/2}s/(\alpha_{i,n}\xi_{i,n})) \left[
\phi\big(n^{1/2}(xs/(\alpha_{i,n}\xi_{i,n})+s\eta_{i,n})\big)
\ind\left(xs/\alpha_{i,n}+\theta_{i}/\sigma > 0 \right)\right.\\
&+\left. \phi\big(n^{1/2}(xs/(\alpha_{i,n}\xi_{i,n})-s\eta_{i,n})\big)
\ind\left(xs/\alpha_{i,n}+\theta_{i}/\sigma < 0 \right)\right]
\rho_{n-k}(s) \, ds dx
\end{split}
\end{align}
in case $\theta_i \neq 0$.
\end{proposition}
\begin{figure}[htb]
\begin{center}
\includegraphics[width=6cm]{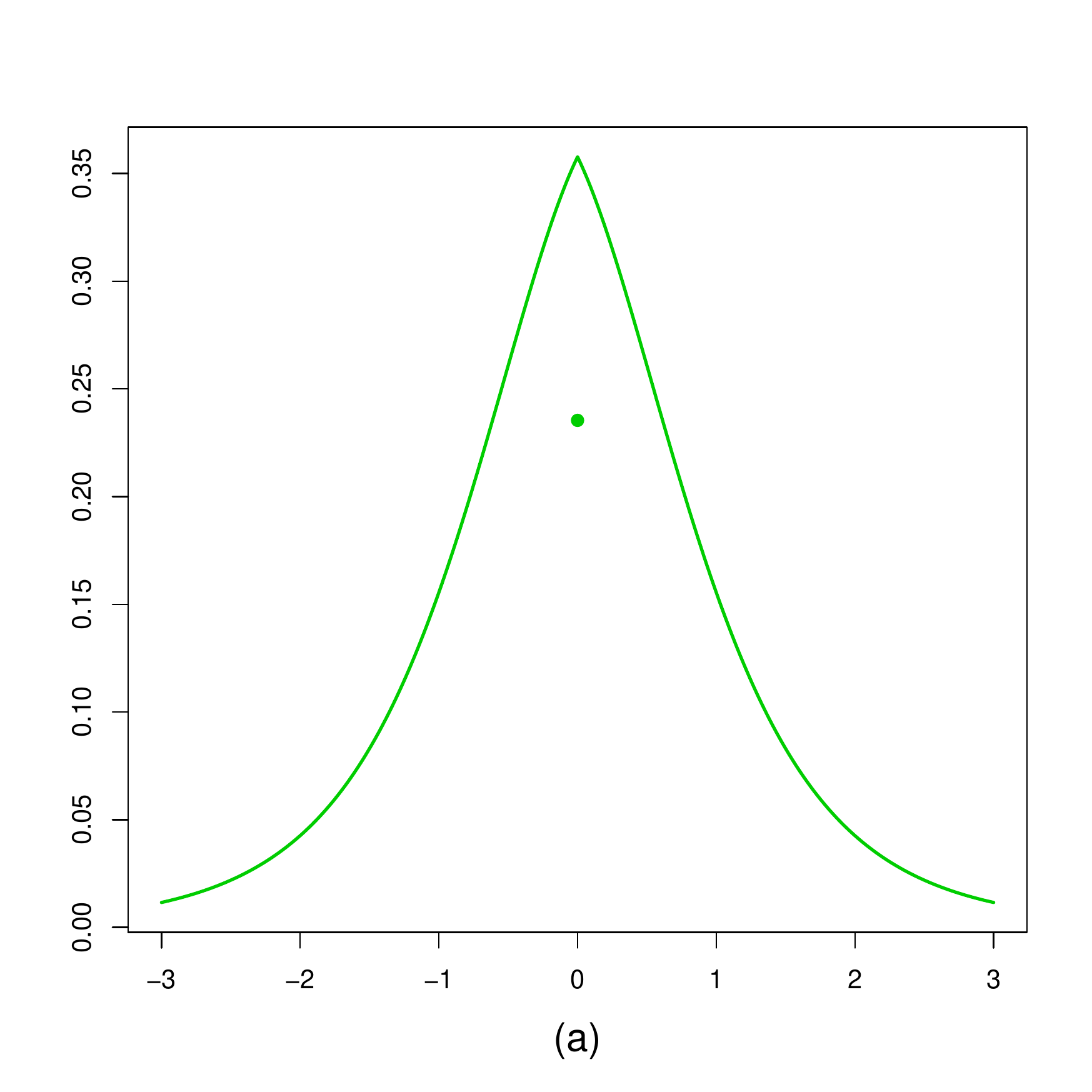}
\includegraphics[width=6cm]{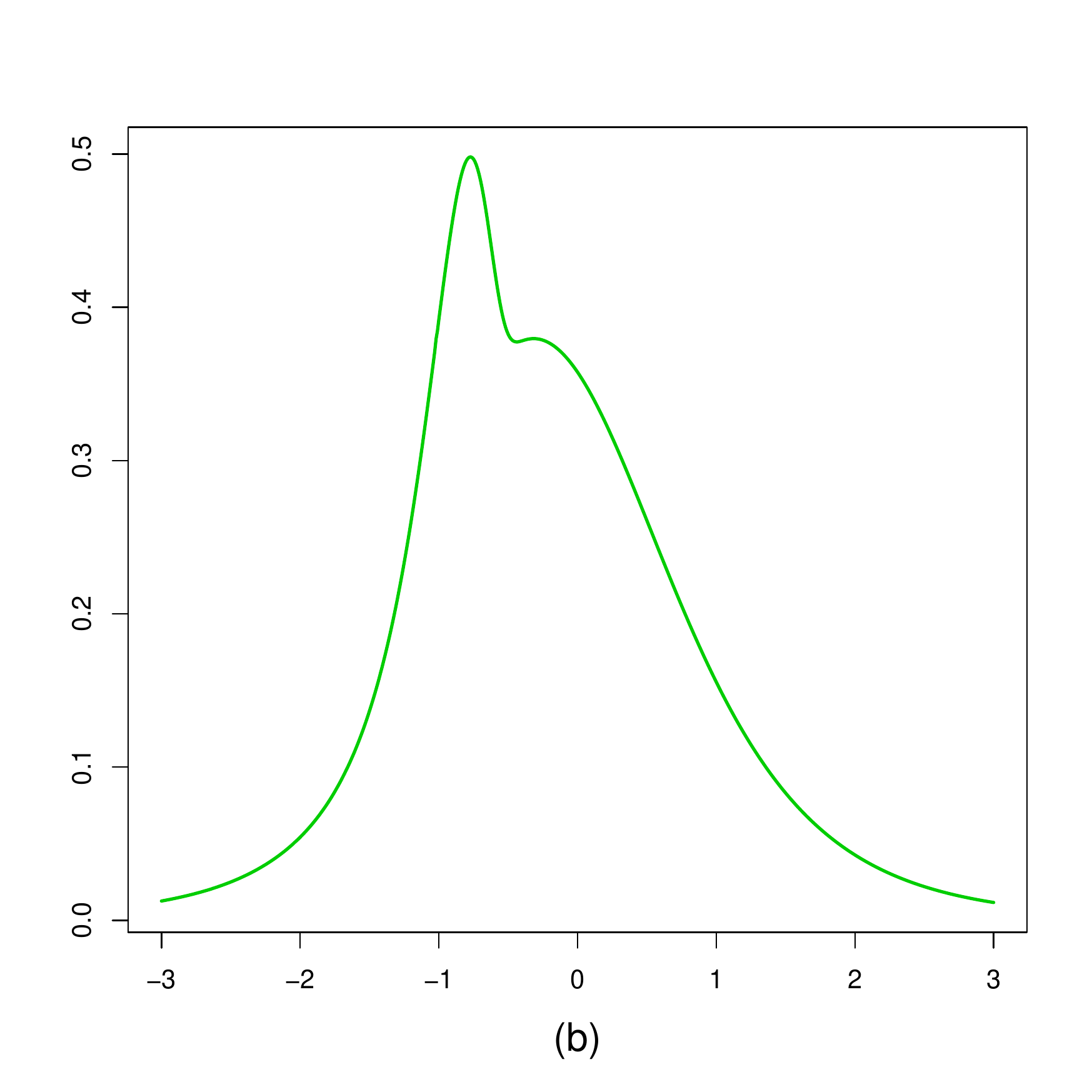}
\caption{\label{plotpdfS} Soft-thresholding: Plots of $d\tilde
  H^i_{S,n,\theta,\sigma}$ for (a) $\theta_i=0$ and (b)
  $\theta_i=0.16$. Exemplarily, in both plots we set to $n=40$,
  $k=35$, $\eta_{i,n}=0.05$, $\xi_{i,n}=1$, $\sigma^2=1$,
  $\alpha_{i,n}=n^{1/2}/\xi_{i,n}$. The dot in part (a) corresponds to
  the total mass of the atomic part.}
\end{center}
\end{figure}
We next consider the adaptive soft-thresholding estimator.
\begin{proposition}[Adaptive soft-thresholding in finite samples]
\label{fsdistAS}
The cdf $\tilde H^i_{AS,n,\theta,\sigma} := \tilde
H^i_{AS,\eta_{i,n},n,\theta,\sigma}$ of $\hat\sigma^{-1}\alpha_{i,n}
(\tilde\theta_{AS,i} - \theta_i)$ is given by
\begin{align}
\label{fscdfAS}
\begin{split}
\tilde H^i_{AS,n,\theta,\sigma}(x) \; = \; & \int_0^\infty 
\left[\Phi\big(z_{n,\theta,\sigma}^{(2)}(xs,s\eta_{i,n})\big) 
\ind\left( xs/\alpha_{i,n} + \theta_{i}/\sigma \geq 0\right)\right. \\
& + \left.\Phi\big(z_{n,\theta,\sigma}^{(1)}(xs,s\eta _{i,n})\big) 
\ind\left(xs/\alpha_{i,n} + \theta_i/\sigma < 0\right)\right]\rho_{n-k}(s)\,ds,
\end{split}
\end{align}
\end{proposition}
or equivalently, its measure is given by
\begin{align}
\label{fspdfAS0}
\begin{split}
d&\tilde H^i_{AS,n,\theta,\sigma}(x) = 
\int_0^\infty \left[\Phi\big(n^{1/2}s\eta_{i,n}\big) - 
\Phi\big(-n^{1/2}s\eta_{i,n}\big) \right] \rho_{n-k}(s)\,ds\,d\delta_0(x) \\
& + \int_0^\infty 0.5 (n^{1/2}s/(\alpha_{i,n}\xi_{i,n})) \left[ 
\phi\big(z_{n,\theta,\sigma}^{(2)}(xs,s\eta_{i,n})\big)
(1 + t_{n,\theta,\sigma}(xs,s\eta_{i,n}))\ind\left(x > 0\right) \right. \\
& + \left. \phi\big(z_{n,\theta,\sigma}^{(1)}(xs,s\eta_{i,n})\big)
(1 - t_{n,\theta,\sigma}(xs,s\eta_{i,n}))\ind\left(x < 0\right) \right] 
\rho_{n-k}(ds) \, dx
\end{split}
\end{align}
for $\theta_i=0$ and by
\begin{align}
\label{fspdfASnot0}
\begin{split}
d&\tilde H^i_{AS,n,\theta,\sigma}(x) = 
\; \ind\{-\sign(\theta_i)\,x \geq 0\} \, \alpha_{i,n}|\theta_i|/(\sigma x^2)
\, \rho_{n-k}\left(-\alpha_{i,n}\theta_i/(\sigma x)\right) \\
& \times \left[\Phi\big(-(n^{1/2}\theta_i/(\sigma\xi_{i,n}))
(1 + \alpha_{i,n}\xi_{i,n}\eta_{i,n}/x)\big) - 
\Phi\big(-(n^{1/2}\theta_i/(\sigma\xi_{i,n}))
(1 - \alpha_{i,n}\xi_{i,n}\eta_{i,n}/x)\big)\right] \, dx\\
& + \int_0^\infty 0.5 (n^{1/2}s/(\alpha_{i,n}\xi_{i,n})) \left[ 
\phi\big(z_{n,\theta,\sigma}^{(2)}(xs,s\eta_{i,n})\big)
(1 + t_{n,\theta,\sigma}(xs,s\eta_{i,n})) 
\ind\left(xs/\alpha_{i,n} + \theta_i/\sigma > 0\right) \right. \\
& + \left. \phi\big(z_{n,\theta,\sigma}^{(1)}(xs,s\eta_{i,n})\big)
(1 - t_{n,\theta,\sigma}(xs,s\eta_{i,n}))
\ind\left(xs/\alpha_{i,n} + \theta_i/\sigma < 0\right)\right] \rho_{n-k}(ds)\,dx
\end{split}
\end{align}
for $\theta_i \neq 0$, where $z_{n,\theta,\sigma}^{(1)}(u,v) \leq
z_{n,\theta,\sigma}^{(2)}(u,v)$ are defined by
$$
0.5n^{1/2} (u/\alpha_{i,n} - \theta_{i}/\sigma)/\xi_{i,n}  \pm n^{1/2}
\sqrt{\left(0.5 (u/\alpha_{i,n} + \theta_{i}/\sigma)/\xi_{i,n} \right)^2 + v^2},
$$
and $t_{n,\theta,\sigma}(u,v) = 0.5 \xi_{i,n}^{-1} 
\left(u/\alpha_{i,n} + \theta_i/\sigma \right)/
\left((0.5\xi_{i,n}^{-1}\big(u/\alpha_{i,n} + \theta_i/\sigma)^2 + 
v^2\right)^{1/2}$.
\begin{figure}[htb]
\begin{center}
\includegraphics[width=6cm]{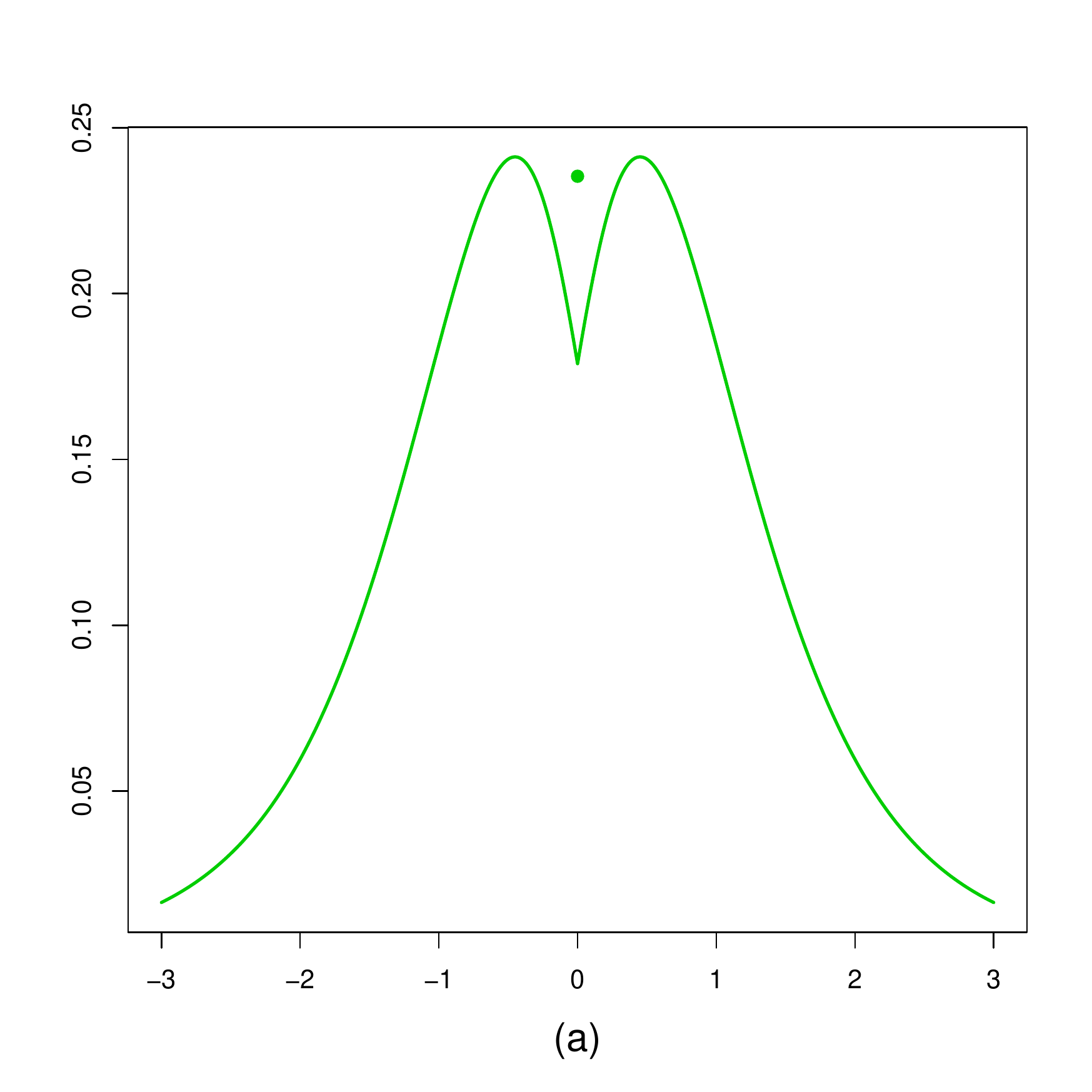}
\includegraphics[width=6cm]{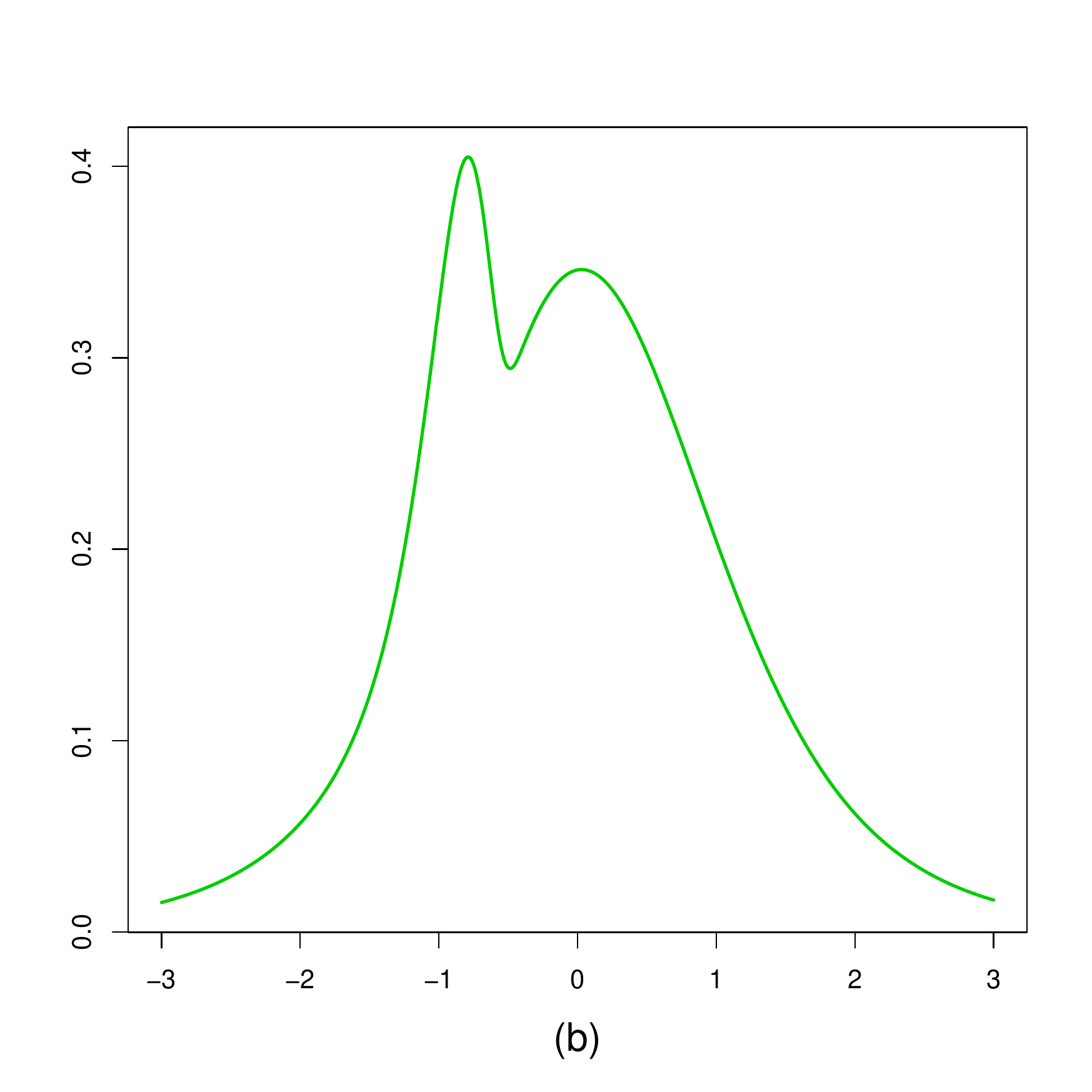}
\caption{\label{plotpdfAS} Adaptive Soft-thresholding: Plots of
  $d\tilde H^i_{AS,n,\theta,\sigma}$ for (a) $\theta_i=0$ and (b)
  $\theta_i=0.16$. Exemplarily, in both plots we set to $n=40$,
  $k=35$, $\eta_{i,n}=0.05$, $\xi_{i,n}=1$, $\sigma^2=1$,
  $\alpha_{i,n}=n^{1/2}/\xi_{i,n}$. The dot in part (a) corresponds to
  the total mass of the atomic part.}
\end{center}
\end{figure}

\bigskip

\noindent
Propositions~\ref{fsdistH}-\ref{fsdistAS} show the following about the
finite-sample distributions of $\hat\sigma^{-1}\alpha_{i,n}(\tilde
\theta_i - \theta_i)$. The distributions of all three estimators are
non-normal and depend on the unknown parameter vector $\theta$ only
through the $i$-th component $\theta_i$. In case $\theta_i = 0$, they
are made up of two components, one being pointmass at 0 and the other
one being absolutely continuous (with respect to Lebesgue-measure)
with a density that is generally bimodal. The pointmass part has the
same weight for all three estimators and is represented by the first
term in (\ref{fspdfH0}), (\ref{fspdfS0}), and (\ref{fspdfAS0}) ending
with $d\delta_0(x)$. The remaining second term, ending with $dx$,
represents the absolutely continuous part. In case $\theta_i \neq 0$,
the distributions are always absolutely continuous, as can be seen in
in (\ref{fspdfHnot0}), (\ref{fspdfSnot0}), and (\ref{fspdfASnot0}).

Propositions~23-25 in \cite{poesch11} show that the distributions of
$\sigma^{-1}\alpha_{i,n}$ $\times(\tilde \theta_i - \theta_i)$ with
non-random scaling consist of pointmass and an absolutely continuous
part for all values of $\theta_i$. This means that scaling the
estimators by $\hat\sigma^{-1}$ instead of $\sigma^{-1}$ results in
smoothing the distribution function to the extent that when $\theta_i
\neq 0$, the previously existing pointmass is spread onto $\R_-$ or
$\R_+$ in case $\theta_i > 0$ or $\theta_i < 0$, respectively,
yielding a continuous distribution function for non-zero values of
$\theta_i$. In case $\theta_i=0$, the atomic part of the distributions
of $\sigma^{-1}\alpha_{i,n}(\tilde \theta_i - \theta_i)$ remains the
same as in the case of scaling by $\hat\sigma^{-1}$.

Curiously, on the other hand, scaling by $\hat\sigma^{-1}$ cancels out
some smoothing effect of the unknown-variance hard-thresholding
estimator $\hat\theta_{H,i}$ in case $\theta_i=0$: there, the density
of the absolutely continuous part of the distribution is only
piecewise continuous (has an excised part) both in the case of known
variance and unknown variance with random $\hat\sigma^{-1}$-scaling,
that is, for the distributions of
$\sigma^{-1}\alpha_{i,n}(\hat\theta_{H,i} -\theta_i)$ and
$\hat\sigma^{-1}\alpha_{i,n}(\tilde\theta_{H,i} -\theta_i)$, but not
so for the case of unknown variance with non-random
$\sigma^{-1}$-scaling, that is, for
$\sigma^{-1}\alpha_{i,n}(\tilde\theta_{H,i} -\theta_i)$, cf.\ Figure~1
in \cite{poelee09} and the paragraph above Remark~26 in
\cite{poesch11}.

\section{Auxiliary Results: Large-Sample Distributions}
\label{lsdist}

We next derive the asymptotic distributions of
$\hat\sigma^{-1}\alpha_{i,n}(\tilde\theta_i - \theta_i)$ under a
so-called moving-parameter framework where the unknown parameter
$\theta$ is allowed to vary with sample size. It is well known that
asymptotics based on a fixed-parameter framework only do not yield a
representative picture of the finite-sample behavior and performance
of the estimator, as is mentioned in the introduction in
Section~\ref{intro}. In particular, we will need the asymptotic
distributions of $\hat\sigma^{-1}\alpha_{i,n}(\tilde\theta_i -
\theta_i)$ under a \emph{moving-parameter} framework in order to carry
out a \emph{uniform} analysis of the coverage properties of confidence
intervals based on $\tilde\theta_i$ in Section~\ref{confunknown}.

The sequence of scaling factors $\alpha_{i,n}$ is chosen according to
the uniform convergence rate of the estimators, that is, $\alpha_{i,n}
= n^{1/2}/\xi_{i,n}$ in the case of conservative tuning, and
$\alpha_{i,n} = (\xi_{i,n}\eta_{i,n})^{-1}$ in case of consistent
tuning, see Theorem~16 in \cite{poesch11}.

Since the number of parameters $k$ may vary with sample size $n$,
asymptotically, we distinguish the case where the number of degrees of
freedom $n-k$ in the variance estimation tends to infinity, and the
case where $n-k$ converges to a finite limit, implying that it
eventually remains constant. For the case where $n-k \to \infty$, note
that $\hat \sigma/\sigma$ converges to 1 in probability, so that the
asymptotic distributions of
$\hat\sigma^{-1}\alpha_{i,n}(\tilde\theta_i - \theta_i)$ coincide with
the asymptotic distributions of
$\sigma^{-1}\alpha_{i,n}(\tilde\theta_i - \theta_i)$ which have been
derived in \cite{poesch11}. 

Note that technically, for the proofs of Section~\ref{confs}, we make
explicit use of the large-sample distributions in the consistently
tuned case only. We state the distributions of the conservatively
tuned case for completeness, but only list the case when $n-k$
converges to a finite limit. For the consistently tuned estimators, we
also quote the case when $n-k \to \infty$ since the results are
directly relevant for Section~\ref{confs}.




\subsection{Conservative Tuning}

We first consider the case where the estimators are tuned to perform
conservative variable selection implying that the sequence of tuning
parameters satisfies $n^{1/2}\eta_{i,n} \to e_i$ with $0 \leq e_i <
\infty$. We start by looking at the hard-thresholding estimator and
derive the asymptotic distribution of
$\hat\sigma^{-1}(n^{1/2}/\xi_{i,n})(\tilde\theta_{H,i} - \theta_i)$.


\begin{proposition}[Hard-thresholding with conservative tuning]
\label{lsdistHconserv}
Suppose that for a given $i \geq 1$ satisfying $i \leq k = k(n)$ for
large enough $n$ we have $\xi_{i,n}\eta_{i,n} \to 0$ and
$n^{1/2}\eta_{i,n} \to e_i$ where $0 \leq e_i < \infty$. Set the
scaling factor $\alpha_{i,n} = n^{1/2}/\xi_{i,n}$. Suppose that the
true parameters $\theta^{(n)} = (\theta_{1,n},\dots,\theta_{k_{n},n})
\in \R^{k_{n}}$ and $\sigma_n \in (0,\infty )$ satisfy
$n^{1/2}\theta_{i,n}/(\sigma_n\xi_{i,n}) \to \nu_i \in \Rquer$ and
that $n-k \to m$ in $\N$. Then $\tilde
H^i_{H,n,\theta^{(n)},\sigma_n}$ converges weakly to the distribution
with cdf
\begin{align}
\label{lscdfHnspm}
\begin{split}
\int_0^\infty & \big[ \Phi\left(xs\right) \ind\left(|xs+\nu_i| > se_i\right)
 + \Phi\left(-\nu_i + se_i\right) \ind\left(0 \leq xs+\nu_i \leq se_i\right) \\ 
& + \Phi\left(-\nu_i-se_i\right) \ind\left(-se_i \leq xs+\nu_i < 0\right)
\big] \rho_m(s)\,ds.
\end{split}
\end{align}
The corresponding measure is given by 
\begin{align*}
\bigg\{ \ind\{-\sign(\nu_i)&\,x \geq 0\} \, |\nu_i|/x^2 \,
\rho_m\left(-\nu_i/x\right) \big[\Phi\left(-\nu_i(1+e_i/x)\right) -
  \Phi\left(-\nu_i(1-e_i/x)\right)\big] \\ 
& +\left. \int_0^\infty s \phi\left(xs\right)\ind\left(|xs + \nu_i| > se_i 
\right) \rho_m(s) \,ds \right\} dx
\end{align*}
for $0 < |\nu_i| < \infty$ and by
$$
\big[T_m(e_i) - T_m(-e_i)\big] d\delta_0(x) + 
\ind\left(|x| > e_i\right) t_m(x)\,dx
$$
for $\nu_i = 0$.

\end{proposition}
Note that the limiting distribution in the above proposition is a
$t$-distribution with $m$ degrees of freedom in case $|\nu_i|=\infty$
or $e_i=0$. Moreover, if $\nu_i=0$, the expression in
(\ref{lscdfHnspm}) simplifies to
$$
T_m(x) \ind\left(|x| > e_i\right) + T_m\left(e_i\right)
\ind\left(0 \leq x \leq e_i\right) 
+ T_m\left(-e_i\right) \ind\left(-e_i \leq x < 0\right).
$$ 
By the argument given above, the limiting distribution for the case
when $n-k \to \infty$ is the same as in Theorem~33(b) in
\cite{poesch11}.


We now look at the soft-thresholding estimator and derive the
asymptotic distribution of
$\hat\sigma^{-1}(n^{1/2}/\xi_{i,n})(\tilde\theta_{S,i} - \theta_i)$.

\begin{proposition}[Soft-thresholding with conservative tuning]
\label{lsdistSconserv}
Suppose that for a given \linebreak $i \geq 1$ satisfying $i \leq k =
k(n)$ for large enough $n$ we have $\xi_{i,n}\eta_{i,n} \to 0$ and
$n^{1/2}\eta_{i,n} \to e_i$ where $0 \leq e_i < \infty$. Set the
scaling factor $\alpha_{i,n} = n^{1/2}/\xi_{i,n}$. Suppose that the
true parameters $\theta^{(n)} = (\theta_{1,n},\dots,\theta_{k_{n},n})
\in \R^{k_{n}}$ and $\sigma_n \in (0,\infty )$ satisfy
$n^{1/2}\theta_{i,n}/(\sigma_n\xi_{i,n}) \to \nu_i \in \Rquer$ and
that $n-k \to m$ in $\N$. Then $\tilde
H^i_{S,n,\theta^{(n)},\sigma_n}$ converges weakly to the distribution
with cdf
\begin{equation}
\label{lscdfSnspm}
\int_0^\infty \big[ \Phi\left((x+e_i)s\right) \ind\left(xs+\nu_i \geq 0\right)
 + \Phi \left((x-e_i)s\right) \ind\left(xs + \nu_i < 0\right) \big] 
\rho_m(s)\,ds.
\end{equation}
The corresponding measure is given by
\begin{align*}
\bigg\{ & \ind\{-\sign(\nu_i)\,x \geq 0\} \, |\nu_i|/x^2 \,
\rho_m\left(-\nu_i/x\right) \big[\Phi\left(-\nu_i(1+e_i/x)\right) -
  \Phi\left(-\nu_i(1-e_i/x)\right)\big] \\ 
& +\left.
\int_0^\infty s\big[\phi\left((x+e_i)s\right)\ind\left(xs + \nu_i > 0\right) +
\phi\left((x-e_i)s\right)\ind\left(xs + \nu_i < 0\right)\big]
\rho_m(s) \,ds \right\} dx
\end{align*}
for $0 < |\nu_i| < \infty$ and by
\begin{equation*}
\big[T_m(e_i) - T_m(-e_i)\big]
d\delta_0(x) + t_m(x+\sign(x)e_i)\,dx
\end{equation*}
for $\nu_i = 0$.
\end{proposition}
Note that the limiting distribution is a $t$-distribution with $m$
degrees of freedom in case $e_i=0$, and a shifted $t$-distribution
with $m$ degrees of freedom in case $|\nu_i| = \infty$ (with shift
$-\sign(\nu_i) e_i$). In case $\nu_i=0$, the expression in
(\ref{lscdfSnspm}) simplifies to
$$
T_m(x+e_i) \ind(x \geq 0) + T_m(x-e_i) \ind(x < 0).
$$ 
By the argument given above, the limiting distribution for the case
$n-k \to \infty$ is the same as in Theorem~34(b) in \cite{poesch11}.


Finally, we consider the adaptive soft-thresholding estimator and
derive the asymptotic distribution of
$\hat\sigma^{-1}(n^{1/2}/\xi_{i,n})(\tilde\theta_{AS,i} - \theta_i)$.

\begin{proposition}[Adaptive soft-thresholding with conservative tuning]
\label{lsdistASconserv}
$\;$ Suppose that for a given $i \geq 1$ satisfying $i \leq k = k(n)$
for large enough $n$ we have $\xi_{i,n}\eta_{i,n} \to 0$ and
$n^{1/2}\eta_{i,n} \to e_i$ where $0 \leq e_i < \infty$. Set the
scaling factor $\alpha_{i,n} = n^{1/2}/\xi_{i,n}$. Suppose that the
true parameters $\theta^{(n)} = (\theta_{1,n},\dots,\theta_{k_{n},n})
\in \R^{k_{n}}$ and $\sigma_n \in (0,\infty )$ satisfy
$n^{1/2}\theta_{i,n}/(\sigma_n\xi_{i,n}) \to \nu_i \in \Rquer$ and
that $n-k \to m$ in $\N$. Then $\tilde
H^i_{AS,n,\theta^{(n)},\sigma_n}$ converges weakly to the distribution
with cdf
\begin{equation}
\label{lscdfASnspm}
\int_0^\infty \!
\big[\Phi(\bar z^{(2)}_{\nu_i}(xs,se_i))\ind(xs+\nu_i \geq 0) +
 \Phi(\bar z^{(1)}_{\nu_i}(xs,se_i))\ind(xs+\nu_i < 0)\big] 
\rho_m(s)ds,
\end{equation}
where $\bar z_\nu^{(1)}(u,v) \leq \bar z_\nu^{(2)}(u,v)$ are defined as
$$
(u-\nu)/2 \pm \sqrt{((u+\nu)/2)^2 + v^2}.
$$ 
The corresponding measure is given by
\begin{align*}
\bigg\{ & \ind\{-\sign(\nu_i)\,x \geq 0\} \, |\nu_i|/x^2 \,
\rho_m\left(-\nu_i/x\right) \big[\Phi\left(-\nu_i(1+e_i/x)\right) -
\Phi\left(-\nu_i(1-e_i/x)\right)\big] \\ 
& + \int_0^\infty 0.5s \big[
\phi\big(\bar z^{(2)}_{\nu_i}(xs,se_i)\big)
\left(1 + \bar t_{\nu_i}(xs,se_i)\right)\ind\left(xs+\nu_i > 0\right) \\
& + \phi\big(\bar z^{(1)}_{\nu_i}(xs,se_i)\big)
\left(1 - \bar t_{\nu_i}(xs,se_i)\right)\ind\left(xs+\nu_i < 0\right)\big]
\rho_m(s)ds\bigg\}\,dx
\end{align*}
for $0 < |\nu_i| < \infty$, where
$$
\bar t_\nu(u,v) := \frac{(u-\nu)/2}{\sqrt{((u+\nu)/2)^2+v^2}}
$$
and by
\begin{align*}
\big[T_m&(e_i) - T_m(-e_i)\big] d\delta_0(x) \\ 
& + \; 0.5 \big[\ind\left(x > 0\right) 
t_m\big(x/2+\sqrt{x^2/4+e_i^2}\big)
\big(1+\frac{x/2}{\sqrt{x^2/4+e_i^2}}\big) \\
& + \; \ind\left(x < 0\right) t_m\big(x/2-\sqrt{x^2/4+e_i^2}\big)
\big(1-\frac{x/2}{\sqrt{x^2/4+e_i^2}}\big)
\big]\,dx
\end{align*}
for $\nu_i = 0$.
\end{proposition}
Note that the limiting distribution in the above proposition is a
$t$-distribution with $m$ degrees of freedom in case $|\nu_i|=\infty$
or $e_i=0$. Moreover, when $\nu_i=0$, the expression in
(\ref{lscdfASnspm}) simplifies to
$$
\ind\left(x \geq 0\right) T_m\big(x/2 + \sqrt{x^2/4 +e_i^2}\big)
+ \ind\left(x < 0\right) T_m\big(x/2 - \sqrt{x^2/4 +e_i^2}\big).
$$ 
By the argument given above, the limiting distribution for the case
$n-k \to \infty$ is the same as in Theorem~35(b) in \cite{poesch11}.

More generally,
Propositions~\ref{lsdistHconserv}-\ref{lsdistASconserv} show the
following: In case that $n-k$ is eventually constant, the large-sample
distribution perfectly captures the behavior of the finite-sample
distribution. In fact, the limiting distribution has the same
functional form as the finite-sample distribution, only with the
quantities $n^{1/2}\theta_{i,n}/\xi_{i,n}$ and $n^{1/2}\eta_{i,n}$
having settled down to their limiting values, $\nu_i$ and $e_i$,
respectively. Similar to the finite-sample case, the large-sample
distribution possesses an atomic part only in case $\nu_i=0$.

In case $n-k \to \infty$, as can be learnt the Theorems 33(b)-35(b)
form \cite{poesch11}, not too surprisingly, the effect of having to
estimate the error variance washes out, and the limiting distribution
actually coincides with the known-variance case,
cf.\ Propositions~27-29 in \cite{poesch11}. Here, the limiting
distributions also possess an atomic part when $\nu_i \neq 0$, more
concretely, the limiting distribution is only absolutely continuous
when $|\nu_i| = \infty$ or $e_i =0$.

\subsection{Consistent Tuning}

We now turn to the case where the estimators are tuned to perform
consistent variable selection so that the sequence of tuning
parameters satisfies $n^{1/2}\eta_{i,n} \to \infty$. We start by
looking at the hard-thresholding estimator and derive the asymptotic
distribution of
$(\hat\sigma\xi_{i,n}\eta_{i,n})^{-1}(\tilde\theta_{H,i} - \theta_i)$.


\begin{proposition}[Hard-thresholding with consistent tuning]
\label{lsdistHconsist}
$\;$ Suppose that for a given $i \geq 1$ satisfying $i \leq k = k(n)$
for large enough $n$ we have $\xi_{i,n}\eta_{i,n} \to 0$ and
$n^{1/2}\eta_{i,n} \to \infty$. Set the scaling factor $\alpha_{i,n} =
(\xi_{i,n}\eta_{i,n})^{-1}$. Suppose that the true parameters
$\theta^{(n)} = (\theta_{1,n},\dots,\theta_{k_{n},n}) \in \R^{k_{n}}$
and $\sigma_n \in (0,\infty )$ satisfy
$\theta_{i,n}/(\sigma_n\xi_{i,n}\eta_{i,n}) \to \zeta_i \in \Rquer$.

\begin{enumerate}

\item If $n - k$ eventually converges to a finite limit $m \in \N$, then 

\begin{enumerate}

\item $|\zeta_i| = \infty$ or $\zeta_i = 0$ implies that $\tilde
  H^i_{H,n,\theta^{(n)},\sigma_n}$ converges weakly to $\delta_0$.

\item $0 < \zeta_i < \infty$ implies that $\tilde
  H^i_{H,n,\theta^{(n)},\sigma_n}$ converges weakly to the
  distribution with cdf
$$ 
\ind(-1 \leq x < 0) \Pr(m \zeta_i^2 < \chi_m^2 \leq m \zeta_i^2/x^2) + 
\ind(x \geq 0).
$$

\item $-\infty < \zeta_i < 0$ implies that $\tilde
  H^i_{H,n,\theta^{(n)},\sigma_n}$ converges weakly to the
  distribution with cdf
$$ 
\ind(0 \leq x <1)\big[\Pr(\chi_m^2 \leq m \zeta_i^2) + 
\Pr(\chi_m^2 > m \zeta_i^2/x^2)\big] + \ind(x \geq 1).
$$

\end{enumerate}

\item If $n-k \to \infty$, then

\begin{enumerate}

\item $|\zeta_i| < 1$ implies that $\tilde
  H^i_{H,n,\theta^{(n)},\sigma_n}$ converges weakly to
  $\delta_{-\zeta_i}$.

\item $|\zeta_i| > 1$ implies that $\tilde
  H^i_{H,n,\theta^{(n)},\sigma_n}$ converges weakly to $\delta_0$.

\item $|\zeta_i| = 1$ implies that $\tilde
  H^i_{H,n,\theta^{(n)},\sigma_n}$ converges weakly
$$
w(f_i,r_i,s_i) \delta_{-\zeta_i} + (1 - w(f_i,r_i,s_i)) \delta_0,
$$ 
where $f_{i,n} = n^{1/2}\eta_{i,n}/(n-k)^{1/2} \to f_i \in \Rquer$,
$r_{i,n} = n^{1/2}(\eta_{i,n}-\zeta_i
\theta_{i,n}/(\sigma_n\xi_{i,n})) \to r_i \in \Rquer$ and
$r_{i,n}/f_{i,n} \to s_i \in \Rquer$ and the weight function $w$ is given by
$$
w(f_i,r_i,s_i) = \begin{cases} 
\Phi(r_i) & f_i = 0, \\
\int_{\R} \Phi(2^{-1/2} f_i t + r_i) \phi(t)dt & 0 < f_i < \infty, \\
\Phi(2^{1/2}s_i) & f_i = \infty.
\end{cases}
$$

\end{enumerate}

\end{enumerate}

\end{proposition}
We now look at the soft-thresholding estimator and derive the
asymptotic distribution of
$(\hat\sigma\xi_{i,n}\eta_{i,n})^{-1}(\tilde\theta_{S,i} - \theta_i)$.

\begin{proposition}[Soft-thresholding with consistent tuning] 
\label{lsdistSconsist}
Suppose that for a given $i \geq 1$ satisfying $i \leq k = k(n)$ for
large enough $n$ we have $\xi_{i,n}\eta_{i,n} \to 0$ and
$n^{1/2}\eta_{i,n} \to \infty$. Set the scaling factor $\alpha_{i,n} =
(\xi_{i,n}\eta_{i,n})^{-1}$. Suppose that the true parameters
$\theta^{(n)} = (\theta_{1,n},\dots,\theta_{k_{n},n}) \in \R^{k_{n}}$
and $\sigma_n \in (0,\infty )$ satisfy
$\theta_{i,n}/(\sigma_n\xi_{i,n}\eta_{i,n}) \to \zeta_i \in \Rquer$.

\begin{enumerate}

\item If $n-k$ eventually converges to a finite limit $m \in \N$, then

\begin{enumerate}

\item $|\zeta_i| = \infty$ implies that $\tilde
  H^i_{S,n,\theta^{(n)},\sigma_n}$ converges weakly to
  $\delta_{-\sign(\zeta_i)}$.

\item $0 < \zeta_i < \infty$ implies that $\tilde
  H^i_{S,n,\theta^{(n)},\sigma_n}$ converges weakly to the
  distribution with cdf
$$
\ind(-1 \leq x < 0) \Pr(\chi_m^2 \leq m \zeta_i^2/x^2) + \ind(x \geq 0).
$$

\item $\zeta_i = 0$ implies that $\tilde
  H^i_{S,n,\theta^{(n)},\sigma_n}$ converges weakly to $\delta_0$.

\item $-\infty < \zeta_i < 0$ implies that $\tilde
  H^i_{S,n,\theta^{(n)},\sigma_n}$ converges weakly to the
  distribution with cdf
$$
\ind(0 \leq x < 1) \Pr(\chi_m^2 > m \zeta_i^2/x^2) + \ind(x \geq 1).
$$

\end{enumerate}

\item If $n-k \to \infty$, then 

\begin{enumerate}

\item $|\zeta_i| \leq 1$ implies that $\tilde
  H^i_{S,n,\theta^{(n)},\sigma_n}$ converges weakly to
  $\delta_{-\zeta_i}$.

\item $|\zeta_i| > 1$ implies that $\tilde
  H^i_{S,n,\theta^{(n)},\sigma_n}$ converges weakly to
  $\delta_{-\sign(\zeta_i)}$.

\end{enumerate}

\end{enumerate}

\end{proposition}
Finally, we consider the adaptive soft-thresholding estimator
and derive the asymptotic distribution of
$(\hat\sigma\xi_{i,n}\eta_{i,n})^{-1}(\tilde\theta_{AS,i} - \theta_i)$.

\begin{proposition}[Adaptive soft-thresholding with consistent tuning]
\label{lsdistASconsist}
Suppose that for a given $i \geq 1$ satisfying $i \leq k = k(n)$ for
large enough $n$ we have $\xi_{i,n}\eta_{i,n} \to 0$ and
$n^{1/2}\eta_{i,n} \to \infty$. Set the scaling factor $\alpha_{i,n} =
(\xi_{i,n}\eta_{i,n})^{-1}$. Suppose that the true parameters
$\theta^{(n)} = (\theta_{1,n},\dots,\theta_{k_{n},n}) \in \R^{k_{n}}$
and $\sigma_n \in (0,\infty )$ satisfy
$\theta_{i,n}/(\sigma_n\xi_{i,n}\eta_{i,n}) \to \zeta_i \in \Rquer$.

\begin{enumerate}

\item If $n-k$ eventually converges to a finite limit $m \in \N$, then

\begin{enumerate}

\item $\zeta_i = 0$ or $|\zeta_i| = \infty$ implies that $\tilde
  H^i_{AS,n,\theta^{(n)},\sigma_n}$ converges weakly to $\delta_0$.

\item $0 < \zeta_i < \infty$ implies that $\tilde
  H^i_{AS,n,\theta^{(n)},\sigma_n}$ converges weakly to the
  distribution with cdf
$$
\ind(-1 \leq x < 0) \Pr(m \zeta_i^2x^2 < \chi_m^2 \leq m \zeta_i^2/x^2)
+ \ind(x \geq 0).
$$

\item $-\infty < \zeta_i < 0$ implies that $\tilde
  H^i_{AS,n,\theta^{(n)},\sigma_n}$ converges weakly to the
  distribution with cdf
$$
\ind(0 \leq x < 1) \big[\Pr(\chi_m^2 \leq m \zeta_i^2 x^2) + 
\Pr(\chi_m^2 > m \zeta_i^2/x^2)\big] + \ind(x \geq 1).
$$

\end{enumerate}

\item If $n-k \to \infty$, then

\begin{enumerate}

\item $|\zeta_i| \leq 1$ implies that $\tilde
  H^i_{AS,n,\theta^{(n)},\sigma_n}$ converges weakly to
  $\delta_{-\zeta_i}$.

\item $1 < |\zeta_i| < \infty$ implies that $\tilde
  H^i_{AS,n,\theta^{(n)},\sigma_n}$ converges weakly to
  $\delta_{-1/\zeta_i}$.

\item $|\zeta_i| = \infty$ implies that $\tilde
  H^i_{AS,n,\theta^{(n)},\sigma_n}$ converges weakly to $\delta_0$.

\end{enumerate}

\end{enumerate}

\end{proposition}
Inspecting the results from
Propositions~\ref{lsdistHconsist}-\ref{lsdistASconsist}, we find the
following: The limiting distributions are concentrated on the interval
$[-1,1]$ in all cases.

In the case when $n-k \to \infty$, the limiting distributions always
collapse to pointmass in case of soft- and adaptive soft-thresholding
and to a convex combination of two pointmasses in case of
hard-thresholding. The limits are the same as for the known-variance
case for soft- and adaptive soft-thresholding
\citep[cf.\ Propositions~28(b) and 29(b) in][]{poesch11}. It seems
worth noting that for the hard-thresholding estimator, the limit also
depends on how fast $n-k$ diverges in relation to the tuning parameter
$\eta_{i,n}$.

When the number of degrees of freedom $n-k$ eventually converges to a
constant, in contrast to the limiting behavior when scaling the same
estimators by the unknown variance instead of the corresponding
estimator \citep[cf.\ Theorems~36(a)-38(a) in][]{poesch11}, we find
that the limiting distributions of are always concentrated on the
interval $[-1,1]$. Moreover, we find an absolutely continuous part in
the limit also for hard-thresholding, which was a convex combination
of two pointmasses for non-random scaling.

When $n-k \to \infty$, we observe the same interesting phenomenon as
for the estimators in the known-variance case, namely that the
limiting distributions always collapse to pointmass or a convex
combination of two pointmasses. This shows that the consistently tuned
thresholding estimators exhibit a severe bias-problem in the sense
that the ``bias-component'' is the dominant component and of larger
order than the ``stochastic variability'' of the estimator. Note that
this is not due to a ``wrong choice'' of scaling factor since indeed
the scaling was chosen according to the uniform convergence rate of
the estimators. When $n-k$ converges to a finite value, some
``stochastic variability'' due to the variance estimator can still be
seen in the limit, but only to the extent that it is contained in
either the interval $[-1,0]$ or the interval $[0,1]$.

\begin{remark}
Note that \cite{poesch11} show that the distribution functions of
$\sigma^{-1}\alpha_{i,n}(\hat\theta_i - \theta_i)$ and
$\sigma^{-1}\alpha_{i,n}(\tilde\theta_i - \theta_i)$ are uniformly
close when the tuning parameter satisfies
$n^{1/2}\eta_{i,n}/(n-k)^{1/2} \to 0$ (cf.\ Theorem~30 in that
reference). It seems worth noting that it can be shown that the same
is \emph{not} true anymore when comparing the distributions of
$\sigma^{-1}\alpha_{i,n}(\hat\theta_i - \theta_i)$ and
$\hat\sigma^{-1}\alpha_{i,n}(\tilde\theta_i - \theta_i)$. We abstain
from spelling out details.
\end{remark}

\section{Main Results: Confidence Sets}
\label{confs}

We now turn to the main topic of the paper: analyzing coverage
properties of confidence sets based on the thresholding estimators
defined in Section~\ref{model}. More concretely, we consider
confidence intervals for a component $\theta_i$ of the unknown
parameter vector $\theta$ based on the introduced thresholding
estimators. For comparison and illustration purposes, note that the
standard $(1-\alpha)$-confidence interval based on the least-squares
estimator in the known-variance case is given by $[\hat\theta_{LS,i} -
  \sigma z_{i,n},\hat\theta_{LS,i} + \sigma z_{i,n}]$ where $z_{i,n}$
satisfies
$$ 
\Phi(n^{1/2}z_{i,n}/\xi_{i,n}) - \Phi(-n^{1/2}z_{i,n}/\xi_{i,n}) = 
1-\alpha,
$$ 
and by $[\hat\theta_{LS,i} - \hat\sigma t_{i,n},\hat\theta_{LS,i}
  + \hat\sigma t_{i,n}]$, where $t_{i,n}$ solves
$$
T_{n-k}(n^{1/2}t_{i,n}/\xi_{i,n}) - T_{n-k}(-n^{1/2}t_{i,n}/\xi_{i,n}) = 
1-\alpha
$$ 
in the unknown-variance case. Note that the coverage probabilities in
the two displays above do not depend on $\theta_i$ and $\sigma$ so
that the coverage is the same for any values of the unknown
parameters. For the thresholding estimators, however, this is clearly
not the case anymore as can be seen from the corresponding expressions
for the finite-sample distributions that depend on $\theta_i$ and
$\sigma$ in a complicated manner. Therefore, in order to come up with
confidence intervals based on the thresholding estimators, we actually
need find lower bounds on the coverage probabilities over all unknown
parameters, that is, we need to minimize (or lower bound) those
probabilities over $\theta_i$ (and $\sigma$). This is done in the
following two sections for the known- and the unknown-variance case.

\subsection{Known-variance Case}
\label{confknown}

We consider confidence intervals of the form $[\hat\theta_i - \sigma
  a_{i,n}, \hat\theta_i + \sigma b_{i,n}]$ where $a_{i,n}, b_{i,n}$
are non-negative real numbers. We are interested in the finite-sample
coverage properties of such intervals, more concretely we aim to
investigate for which $a_{i,n}$ and $b_{i,n}$ we have
$$ 
P_{n,\theta,\sigma}(\theta_i \in [\hat\theta_i - \sigma a_{i,n},
  \hat\theta_i + \sigma b_{i,n}]) \geq 1-\alpha \;\; \text{ for all } \theta
\in \R^k
$$ for some prescribed coverage probability $1-\alpha \in (0,1)$. The
finite-sample results for the estimators in the known-variance case
can be summarized in the following theorem.

\begin{theorem}
\label{confknownfs}
For every $n \geq 1$ and every $\alpha$ satisfying $0<\alpha<1$ we have:

\begin{enumerate}

\item The infimal coverage probability $\inf_{\theta \in
  \R^k}P_{n,\theta,\sigma}(\theta_i \in [\hat\theta_{H,i} - \sigma
  a_{i,n},\hat\theta_{H,i} + \sigma b_{i,n}])$ is given by
$$
\begin{array}{cl}
\Phi(n^{1/2}(a_{i,n}/\xi_{i,n}-\eta_{i,n})) - \Phi(-n^{1/2}b_{i,n}/\xi_{i,n}) 
& \text{ if } \; \xi_{i,n}\eta_{i,n} \leq a_{i,n}+b_{i,n} 
\; \text{ and } \; a_{i,n} \leq b_{i,n} \\ 
\Phi(n^{1/2}(b_{i,n}/\xi_{i,n}-\eta_{i,n})) - \Phi(-n^{1/2}a_{i,n}/\xi_{i,n}) 
& \text{ if } \; \xi_{i,n}\eta_{i,n} \leq a_{i,n}+b_{i,n} 
\; \text{ and } \; a_{i,n} > b_{i,n} \\ 
0 & \text{ if } \; \xi_{i,n}\eta_{i,n} > a_{i,n}+b_{i,n}.
\end{array}
$$ Among all intervals $[\hat\theta_{H,i} - \sigma
  a_{i,n},\hat\theta_{H,i} + \sigma b_{i,n}]$ with infimal coverage
probability not less than $1-\alpha$ there is a unique shortest interval
$C_{H,i,n} = [\hat\theta_{H,i} - \sigma a^*_{H,i,n},\hat\theta_{H,i}
  + \sigma b^*_{H,i,n}]$ characterized by $a^*_{H,i,n} = b^*_{H,i,n}$
with $a^*_{H,i,n}$ being the unique solution in $a$ of
$$
\Phi(n^{1/2}(a/\xi_{i,n} - \eta_{i,n})) - \Phi(-n^{1/2}a/\xi_{i,n}) 
= 1-\alpha.
$$ The interval $C_{H,i,n}$ has infimal coverage probability equal to
$1-\alpha$ and $a^*_{H,i,n}$ satisfies $a^*_{H,i,n} > \xi_{i,n}\eta_{i,n}/2$.

\item The infimal coverage probability $\inf_{\theta \in
  \R^k}P_{n,\theta,\sigma}(\theta_i \in [\hat\theta_{S,i} - \sigma
  a_{i,n},\hat\theta_{S,i} + \sigma b_{i,n}])$ is given by
$$
\begin{array}{ll}
\Phi(n^{1/2}(a_{i,n}/\xi_{i,n} - \eta_{i,n})) - \Phi(n^{1/2}(-b_{i,n}/\xi_{i,n}
 - \eta_{i,n})) & \text{ if } \; a_{i,n} \leq b_{i,n} \\ 
\Phi(n^{1/2}(b_{i,n}/\xi_{i,n} - \eta_{i,n})) - \Phi(n^{1/2}(-a_{i,n}/\xi_{i,n}
 - \eta_{i,n})) & \text{ if } \; a_{i,n} > b_{i,n}.
\end{array}
$$
Among all intervals $[\hat\theta_{S,i} - \sigma
  a_{i,n},\hat\theta_{S,i} + \sigma b_{i,n}]$ with infimal coverage
probability not less than $1-\alpha$ there is a unique shortest interval
$C_{S,i,n} = [\hat\theta_S - \sigma a^*_{S,i,n},\hat\theta_S +
  \sigma b^*_{S,i,n}]$ characterized by $a^*_{S,i,n} = b^*_{S,i,n}$
with $a^*_{S,i,n}$ being the unique solution in $a$ of
$$ 
\Phi(n^{1/2}(a/\xi_{i,n} - \eta_{i,n})) - 
\Phi(n^{1/2}(-a/\xi_{i,n} - \eta_{i,n})) = 1-\alpha.
$$ 
The interval $C_{S,i,n}$\ has infimal coverage probability
equal to $1-\alpha$ and $a^*_{S,i,n}$ is positive.

\item The infimal coverage probability $\inf_{\theta \in
  \R^k}P_{n,\theta,\sigma}(\theta_i \in [\hat\theta_{AS,i} - \sigma
  a_{i,n},\hat\theta_{AS,i} + \sigma b_{i,n}])$ is given by
\begin{align*}
\Phi(n^{1/2}(a_{i,n}/\xi_{i,n} - \eta_{i,n})) - 
\Phi(n^{1/2}\xi_{i,n}^{-1}((a_{i,n} - b_{i,n})/2 - & 
\sqrt{((a_{i,n} + b_{i,n}/2)^2 + \xi^2_{i,n}\eta^2_{i,n}})) \\
 & \text{ if } \; a_{i,n} \leq b_{i,n} \\ 
\Phi(n^{1/2}(b_{i,n}/\xi_{i,n} - \eta_{i,n})) - 
\Phi(n^{1/2}\xi_{i,n}^{-1}((b_{i,n} - a_{i,n})/2 - &
\sqrt{((a_{i,n} + b_{i,n}/2)^2 + \xi^2_{i,n}\eta^2_{i,n}})) \\
 & \text{ if } \; a_{i,n} > b_{i,n}.
\end{align*}
Among all intervals $[\hat\theta_{AS,i} - \sigma
  a_{i,n},\hat\theta_{AS,i} + \sigma b_{i,n}]$ with infimal coverage
probability not less than $1-\alpha $ there is a unique shortest
interval $C_{AS,i,n} = [\hat\theta_{AS,i} - \sigma
  a^*_{AS,i,n},\hat\theta_{AS,i} + \sigma b^*_{AS,i,n}]$ characterized
by $a^*_{AS,i,n}=b^*_{AS,i,n}$ with $a^*_{AS,i,n}$ being the unique
solution in $a$ of
$$ 
\Phi(n^{1/2}(a/\xi_{i,n} - \eta_{i,n})) -
\Phi\left(-n^{1/2}\sqrt{(a/\xi_{i,n})^2 + \eta_{i,n}^2}\right) = 1-\alpha.
$$ 
The interval $C_{A,i,n}$ has infimal coverage probability equal to
$1-\alpha$ and $a^*_{AS,i,n}$ is positive.

\end{enumerate}

\end{theorem}
It seems worth noting that the infimal coverage probabilities in
Theorem~\ref{confknownfs} do not depend on the particular value of
$\sigma$. This is due to the scale-equivariance of the estimators
mentioned in the Section~\ref{model}, more concretely we have that for
$C_{i,n} = [\hat\theta_i - \sigma a_{i,n},\hat\theta_i + \sigma
  a_{i,n}]$ and $p^i_n(\theta,\sigma,a_{i,n},\eta_{i,n}) =
P_{n,\theta,\sigma}(\theta_i \in C_{i,n})$ that
$p^i_n(\theta,\sigma,a_{i,n},\eta_{i,n}) =
p^i_n(\theta/\sigma,1,a_{i,n},\eta_{i,n})$, so that for the infimal
coverage probabilities, it suffices to consider the case $\sigma = 1$.

The fact that symmetric intervals are the shortest even though the
distributions of the estimation errors $\hat\theta_i - \theta_i$ are
not symmetric seems to be linked to the fact that the corresponding
distributions under the parameters $\theta_i$ and $-\theta_i$ are
mirror images of each other.

\subsubsection{Comparison of lengths}
\label{asympknown}

When comparing the lengths of the different confidence intervals, note
that the half-length of the standard $(1-\alpha)$-confidence interval
based on the least-squares estimator is given by $\sigma
z_{i,n} = \sigma(\xi_{i,n}/n^{1/2})\Phi^{-1}(1-\alpha/2)$ and that
the above theorem shows that
$$
z_{i,n} < a^*_{S,i,n} < a^*_{AS,i,n} < a^*_{H,i,n}
$$ 
for each $n \geq 1$ and every $0 < \alpha < 1$. The asymptotic
behavior is as follows. In the regime of conservative variable selection
where $n^{1/2}\eta_{i,n} \to e_i < \infty$, it follows immediately
from Theorem~\ref{confknownfs} that $n^{1/2} a^*_{H,i,n}/\xi_{i,n}$,
$n^{1/2} a^*_{S,i,n}/\xi_{i,n}$, and $n^{1/2} a^*_{AS,i,n}/\xi_{i,n}$
converge to the unique solutions in $a$ of
\begin{align*}
& \Phi(a - e_i) - \Phi(-a) \; = \; 1-\alpha, \\
& \Phi(a - e_i) - \Phi(-a - e_i) \; = \; 1-\alpha, \\
& \Phi(a - e_i) - \Phi(-\sqrt{a^2 + e_i^2}) \; = \; 1-\alpha,
\end{align*}
respectively, so that while $a^*_{H,i,n}$, $a^*_{S,i,n}$, and
$a^*_{AS,i,n}$ are larger than $z_{i,n}$ for each $n \geq 1$, they are
of the same order $\xi_{i,n}/n^{1/2}$. For the case of consistent
variable selection, however, when $n^{1/2}\eta_{i,n} \to \infty$, a
different picture arises. Let $a^*_{i,n}$ stand for any of the
expressions $a^*_{H,i,n}$, $a^*_{S,i,n}$, and
$a^*_{AS,i,n}$. Theorem~\ref{confknownfs} shows that
$$
\Phi(n^{1/2}(a^*_{i,n}/\xi_{i,n} - \eta_{i,n})) \longrightarrow 1-\alpha,
$$ 
since $\Phi(-n^{1/2}a^*_{H,i,n}/\xi_{i,n}) < \Phi(-n^{1/2}\eta_{i,n}/2)
\to 0$, and both $\Phi(n^{1/2}(-a^*_{S,i,n}/\xi_{i,n} - \eta_{i,n}))$
and $\Phi\left(-n^{1/2}((a^*_{AS,i,n}/\xi_{i,n})^2 +
  \eta_{i,n}^2)^{1/2}\right)$ are less than $\Phi(-n^{1/2}\eta_{i,n})$
which converges to 0. This entails that
\begin{equation}
\label{confknownlength}
n^{1/2}a^*_{i,n}/\xi_{i,n} = n^{1/2}\eta_{i,n} + \Phi^{-1}(1-\alpha) + o(1).
\end{equation}
Since $n^{1/2}\eta_{i,n} \to \infty$ and $n^{1/2}z_{i,n}/\xi_{i,n} =
\Phi^{-1}(1-\alpha/2)$, we find the lengths of the confidence
intervals for the thresholding estimators are actually larger than the
standard confidence interval by an order of magnitude in the case of
consistent tuning.

\subsubsection{A simple confidence interval}

For the case of consistent tuning of the estimators, we find that we
can construct a simple confidence interval for which we can derive the
asymptotic coverage probabilities. Consider intervals of the form
$D_{i,n} = [\hat\theta_i - \sigma d\xi_{i,n}\eta_{i,n},\hat\theta_i +
  \sigma d \xi_{i,n}\eta_{i,n}]$, where $d$ is a non-negative real
number. (Note that $\xi_{i,n}\eta_{i,n}$ corresponds to the uniform
convergence rate in the consistently tuned case.)

\begin{proposition}
\label{confknowncrude} 
Suppose that for given $i \geq 1$ satisfying $i \leq k = k(n)$ for
large enough $n$ we have $\xi_{i,n}\eta_{i,n} \to 0$ and
$n^{1/2}\eta_{i,n} \to \infty$. Then
$$
\lim_{n \to \infty} \inf_{\theta \in \R^k} 
P_{n,\theta,\sigma}(\theta_i \in D_{i,n}) = 1
$$
if $d > 1$, 
$$
\;\;\; \lim_{n \to \infty} \inf_{\theta \in \R^k} 
P_{n,\theta,\sigma}(\theta_i \in D_{i,n}) = 1/2
$$
if $d = 1$ and 
$$
\lim_{n \to \infty} \inf_{\theta \in \R^k} 
P_{n,\theta,\sigma}(\theta_i \in D_{i,n}) = 0
$$
if $d < 1$. 
\end{proposition}
Since Proposition~\ref{confknowncrude} provides rather crude
confidence intervals, it may be of interest to asymptotically compare
the lengths with the exact intervals proposed in
Theorem~\ref{confknownfs}. From (\ref{confknownlength}) we find that
$$
d \xi_{i,n}\eta_{i,n}/a^*_{i,n} =  d + o(1),
$$
where $d>1$ in fact can be chosen arbitrarily close to 1. 

\subsection{Unknown-Variance Case}
\label{confunknown}

Computing coverage probabilities becomes more complicated when the
error variance is not known. In finite samples, we derive a closed
form expression for the infimal coverage probability for an interval
based on the soft-thresholding estimator and provide lower bounds
when the interval is based on hard-thresholding or adaptive
soft-thresholding.

We construct intervals with exact asymptotic coverage for all three
estimators when the number of degrees of freedom $n-k$ tends to
infinity fast enough in relation to the tuning parameter (for either
conservative or consistent tuning) based on the fact that in this case
the infimal coverage probabilities in the known- and unknown-variance
case become arbitrarily close asymptotically.

Similarly to Proposition~\ref{confknowncrude} we construct a simple
confidence interval when the estimators are tuned to perform
consistent variable selection, independent of whether and how fast the
number of degrees of freedom tends to infinity.

Based on our findings in Theorem~\ref{confknownfs} in the
known-variance case, for the unknown-variance case we consider
symmetric intervals of the form $E_{i,n} = [\tilde\theta_i -
  \hat\sigma a_{i,n},\tilde\theta_i + \hat\sigma a_{i,n}]$ for
non-negative $a_{i,n}$. Note that 
\begin{eqnarray}
\label{confknownunknown}
P_{n,\theta,\sigma}(\theta_i \in E_{i,n}) &=& 
\int_0^\infty p^i_n(\theta,\sigma,a_{i,n}s,\eta_{i,n}s) \rho_{n-k}(s)\,ds 
\nonumber \\
&=& \int_0^\infty p^i_n(\theta/\sigma,1,a_{i,n}s,\eta_{i,n}s) \rho_{n-k}(s)\,ds
\end{eqnarray}
as discussed in Section~\ref{confknown}, so that again it suffices to consider
the case $\sigma = 1$ and in fact we have
$$
\inf_{\theta \in \R^k,\sigma \in \R_+} P_{n,\theta,\sigma}(\theta_i
\in E_{i,n}) = 
\inf_{\theta \in \R^k} P_{n,\theta,1}(\theta_i \in E_{i,n}).
$$ 
The lower bounds derived in Proposition~\ref{confunknownHfs} and
\ref{confunknownASfs} below are then based on the fact that by we have
(\ref{confknownunknown})
$$ 
\inf_{\theta \in \R^k} P_{n,\theta,1}(\theta_i \in E_{i,n}) \geq
\int_0^\infty \inf_{\theta \in \R^k}
p^i_n(\theta/\sigma,1,a_{i,n}s,\eta_{i,n}s) \rho_{n-k}(s)\,ds,
$$
where the infimum inside the integral on the right-hand side of the
inequality can by computed using Theorem~\ref{confknownfs}.


\subsubsection{Hard-thresholding}

We start by considering the hard-thresholding estimator. We look at
intervals of the form $E_{H,i,n} = [\tilde\theta_{H,i} - \hat\sigma
  a_{i,n},\tilde\theta_{H,i} + \hat\sigma a_{i,n}]$, where $a_{i,n}$
is a non-negative real number. Let $C_{H,i,n} = [\hat\theta_{H,i} -
  \sigma a_{i,n},\hat\theta_{H,i} + \sigma a_{i,n}]$ be the
corresponding interval for the hard-thresholding estimator assuming
knowledge of $\sigma^2$. In contrast to the known-variance case, it is
not possible anymore to come up with a closed-form expression for the
infimal coverage probability of $E_{H,i,n}$. Instead, we uniformly
bound the coverage probability from below therefore allowing for
conservative confidence intervals in finite samples. For the case when
the number of degrees of freedom $n-k$ diverges sufficiently fast in
relation to the tuning parameter $\eta_{i,n}$, we show that the
difference in the infimal coverage probabilities from the known- and
unknown-variance case tends to zero, yielding formulas for confidence
sets in the unknown-variance case that asymptotically have the exact
prescribed coverage. We start by giving the bound in finite samples.

\begin{proposition}
\label{confunknownHfs}
For every $n \geq 2$, we have
\begin{align*}
\inf_{\theta \in \R^k,\sigma \in \R_+} P_{n,\theta,\sigma}(\theta_i
\in E_{H,i,n}) \; \geq \; T_{n-k}(n^{1/2}(a_{i,n}/\xi_{i,n} -
\eta_{i,n})) - T_{n-k}(-n^{1/2}a_{i,n}/\xi_{i,n}).
\end{align*}
\end{proposition}
Note that the bound is trivial in case $a_{i,n} \leq
\xi_{i,n}\eta_{i,n}/2$. The bound is ``asymptotically sharp'' when
$n-k \to \infty$ and $n^{1/2}\eta_{i,n}/(n-k)^{1/2} \to 0$ in the
sense that the difference between the left- and the right-hand side of
the above display then converges to zero. This can be seen from the
following theorem, together with the fact that by
Theorem~\ref{confknownfs}(a) the difference between the corresponding
infimal coverage probability in the known-variance case and the above
bound is less than or equal to $2 \|\Phi - T_{n-k}\|_\infty$, where
$\|\Phi - T_{n-k}\|_\infty = \sup_{x \in \R}|\Phi(x) - T_{n-k}(x)| \to
0$ as $n-k \to \infty$ by Polya's Theorem.

\begin{theorem}
\label{confequivH}
Suppose that for given $i \geq 1$ satisfying $i \leq k = k(n)$ for
large enough $n$ we have $\xi_{i,n}\eta_{i,n} \to 0$, $n-k \to \infty$, and
$n^{1/2}\eta_{i,n}/(n-k)^{1/2} \to 0$. Then, for any sequence
$a_{i,n}$ of non-negative real numbers, we have
$$
\inf_{\theta \in \R^k,\sigma \in \R_+} 
P_{n,\theta,\sigma}(\theta_i \in E_{H,i,n}) - 
\inf_{\theta \in \R^k} 
P_{n,\theta,\sigma}(\theta_i \in C_{H,i,n}) \longrightarrow 0.
$$
\end{theorem}

\begin{remark} 
In case of conservative tuning, the condition
$n^{1/2}\eta_{i,n}/(n-k)^{1/2} \to 0$ in Theorem~\ref{confequivH}
always holds when $n-k \to \infty$. In fact, if $\lim_{n \to \infty}
n^{1/2}\eta_{i,n} = e_i \neq 0$, the condition is equivalent to $n-k
\to \infty$. In case of consistent tuning, $n-k \to \infty$ is a
necessary condition for $n^{1/2}\eta_{i,n}/(n-k)^{1/2} \to 0$ only, and it
means that $n-k$ needs to diverge ``fast enough'' in relation to the
tuning parameter.
\end{remark}


\subsubsection{Soft-thresholding}

We consider intervals of the form $E_{S,i,n} = [\tilde\theta_{S,i} -
  \hat\sigma a_{i,n},\tilde\theta_{S,i} + \hat\sigma a_{i,n}]$ based
on the soft-thresholding estimator, where $a_{i,n}$ is a non-negative
real number. Let $C_{S,i,n} = [\hat\theta_{S,i} - \sigma
  a_{i,n},\hat\theta_{S,i} + \sigma a_{i,n}]$ be the corresponding
interval for the soft-thresholding estimator assuming knowledge of
$\sigma^2$. In contrast to hard-thresholding and adaptive
soft-thresholding, it is indeed possible to derive a closed-form
expression for the infimal coverage probabilities for intervals based
on soft-thresholding also in the unknown-variance case.

\begin{proposition}
\label{confunknownSfs}
For every $n \geq 2$, we have
\begin{align*}
\inf_{\theta \in \R^k,\sigma \in \R_+} 
P_{n,\theta,\sigma}(\theta_i & \in E_{S,i,n})
= T_{n-k}(n^{1/2}(a_{i,n}/\xi_{i,n} - \eta_{i,n})) - 
T_{n-k}(n^{1/2}(-a_{i,n}/\xi_{i,n} - \eta_{i,n})).
\end{align*}
\end{proposition}
We can derive the following theorem, the asymptotic equivalence of the
infimal coverage probabilities in the known- and unknown-variance case
when $n-k \to \infty$, as an immediate consequence of
Proposition~\ref{confunknownSfs}, together with
Theorem~\ref{confknownfs}(b) and the fact that $\|\Phi -
T_{n-k}\|_\infty \to 0$ as $n-k \to \infty$.

\begin{theorem}
\label{confequivS}
Suppose that for given $i \geq 1$ satisfying $i \leq k = k(n)$ for
large enough $n$ we have $\xi_{i,n}\eta_{i,n} \to 0$. Then, for any
sequence $a_{i,n}$ of non-negative real numbers, we have
$$
\inf_{\theta \in \R^k,\sigma \in \R_+} 
P_{n,\theta,\sigma}(\theta_i \in E_{S,i,n}) - 
\inf_{\theta \in \R^k} 
P_{n,\theta,\sigma}(\theta_i \in C_{S,i,n}) \longrightarrow 0 
$$
as $n-k \to \infty$. 
\end{theorem}


\subsubsection{Adaptive soft-thresholding}

We consider intervals of the form $E_{AS,i,n} = [\tilde\theta_{AS,i} -
  \hat\sigma a_{i,n},\tilde\theta_{AS,i} + \hat\sigma a_{i,n}]$ based
on the adaptive soft-thresholding estimator, where $a_{i,n}$ is a
non-negative real number. Let $C_{AS,i,n} = [\hat\theta_{AS,i} -
  \sigma a_{i,n},\hat\theta_{AS,i} + \sigma a_{i,n}]$ be the
corresponding interval in the known-variance case. Just as for the
hard-thresholding estimator, we uniformly bound the coverage
probability from below allowing for conservative confidence intervals
in finite samples. For the case when the number of degrees of freedom
$n-k$ diverges sufficiently fast in relation to the tuning parameter
$\eta_{i,n}$, we show asymptotic equivalence of the infimal coverage
probabilities in the known- and unknown-variance case.

\begin{proposition}
\label{confunknownASfs}
For every $n \geq 2$, we have
\begin{align*}
\inf_{\theta \in \R^k,\sigma \in \R_+} &
P_{n,\theta,\sigma}(\theta_i \in E_{AS,i,n}) \\
& \geq \; T_{n-k}(n^{1/2}(a_{i,n}/\xi_{i,n} - \eta_{i,n})) - 
T_{n-k}(-n^{1/2}\sqrt{(a_{i,n}/\xi_{i,n})^2 + \eta^2_{i,n}}).
\end{align*}
\end{proposition}
This bound is ``asymptotically sharp'' when $n-k \to \infty$ and
$n^{1/2}\eta_{i,n}/(n-k)^{1/2} \to 0$ in the sense that the difference
between the left- and the right-hand side of the above display then
converges to zero. This can be seen from the following theorem
together with the fact that by Theorem~\ref{confknownfs}(c) the
difference between the corresponding infimal coverage probability in
the known-variance case and the above bound is less than or equal to
$2 \|\Phi - T_{n-k}\|_\infty$, where $\|\Phi - T_{n-k}\|_\infty \to 0$
as $n-k \to \infty$.

\begin{theorem}
\label{confequivAS}
Suppose that for given $i \geq 1$ satisfying $i \leq k = k(n)$ for
large enough $n$ we have $\xi_{i,n}\eta_{i,n} \to 0$, $n-k \to \infty$, and
$n^{1/2}\eta_{i,n}/(n-k)^{1/2} \to 0$. Then, for any sequence
$a_{i,n}$ of non-negative real numbers, we have
$$
\inf_{\theta \in \R^k,\sigma \in \R_+} 
P_{n,\theta,\sigma}(\theta_i \in E_{AS,i,n}) - 
\inf_{\theta \in \R^k} 
P_{n,\theta,\sigma}(\theta_i \in C_{AS,i,n}) \longrightarrow 0. 
$$
\end{theorem}
Note that the remark below Theorem~\ref{confequivH} applies for the
above theorem as well.


\subsubsection{Comparison of lengths}

When comparing lengths of confidence intervals in the unknown-variance
case in finite samples, Proposition~\ref{confunknownSfs} and the
inequalities in (\ref{confunknownHub}) and (\ref{confunknownASub}) in
the proofs of Theorems~\ref{confequivH} and \ref{confequivAS}
show that
\begin{equation}
\label{upperbound}
\inf_{\theta \in \R^k, \sigma \in \R_+}P_{n,\theta,\sigma}(\theta_i \in E_{i,n})
\leq  T_{n-k}(n^{1/2}a_{i,n}/\xi_{i,n}) - T_{n-k}(-n^{1/2}a_{i,n}/\xi_{i,n}) 
\end{equation}
holds for $E_{i,n} = [\tilde\theta_i - \hat\sigma
  a_{i,n},\tilde\theta_i + \hat\sigma a_{i,n}]$ based on any of the
three thresholding estimators in consideration. Since the coverage
probability of the standard interval $[\hat\theta_{LS,i} - \hat\sigma
  t_{i,n},\hat\theta_{LS,i} + \hat\sigma t_{i,n}]$ based on the
least-squares estimator is given by
$$
 T_{n-k}(n^{1/2}t_{i,n}/\xi_{i,n}) - T_{n-k}(-n^{1/2}t_{i,n}/\xi_{i,n}),
$$ 
we see that the lengths for the $(1-\alpha)$-intervals based on the
thresholding estimators are larger than the corresponding standard
$(1-\alpha)$-interval in finite samples in the unknown-variance case
also. Asymptotically, a similar discussion as in
Section~\ref{asympknown} applies. When the estimators are
conservatively tuned, the lengths of all intervals including the
standard one are of the same order $\xi_{i,n}/n^{1/2}$ (which can be
seen from the inequality above together with Propositions~14, 16, and
18). In the consistently tuned case, we also find the same behavior as
in the known-variance case in the sense that the lengths of the
intervals based on the thresholding estimators are larger by an order
of magnitude compared to the standard interval (which can be learnt
from Proposition~\ref{confunknowncrude} below using the case $d<1$).


\subsubsection{A simple confidence interval}

Just like in the known-variance case, we can derive a simple
confidence interval based on each of the estimators in the
unknown-variance case when the estimator is tuned to perform
consistent variable selection, in fact, this can be done independently
of whether the number of degrees of freedom $n-k$ remains bounded or
diverges, in particular independently of how fast it diverges in
relation to the tuning parameter. We consider intervals of the form
$F_{i,n} = [\tilde\theta_i - \hat\sigma
  d\xi_{i,n}\eta_{i,n},\tilde\theta_i + \hat\sigma
  d\xi_{i,n}\eta_{i,n}]$ where $d$ is non-negative.

\begin{proposition}
\label{confunknowncrude}
Suppose that for given $i \geq 1$ satisfying $i \leq k = k(n)$ for
large enough $n$ we have $\xi_{i,n}\eta_{i,n} \to 0$ and
$n^{1/2}\eta_{i,n} \to \infty$. Then
$$
\lim_{n \to \infty} \inf_{\theta \in \R^k,\sigma \in \R_+} 
P_{n,\theta,\sigma}(\theta_i \in F_{i,n}) = 1
$$
if $d > 1$, and 
$$
\lim_{n \to \infty} \inf_{\theta \in \R^k,\sigma \in \R_+} 
P_{n,\theta,\sigma}(\theta_i \in F_{i,n}) = 0
$$
if $d < 1$. 
\end{proposition}
\begin{remark}
When $d=1$, we cannot make a comprehensive statement in the
unknown-variance case. However, Proposition~\ref{confunknownSfs} shows
that the limit always equals 1/2 for the soft-thresholding estimator,
and Theorem~\ref{confequivH} and \ref{confequivAS} provide the same
limit for hard-thresholding and adaptive soft-thresholding in case
$n^{1/2}\eta_{i,n}/(n-k)^{1/2} \to 0$.
\end{remark}

\subsubsection{Numerical considerations}

To get an idea of the meaning of the lower bounds in
Propositions~\ref{confunknownHfs} and \ref{confunknownASfs} for hard-
and adaptive soft-thresholding and the upper bound in
(\ref{upperbound}), we provide some numerical considerations. (Note
that for soft-thresholding as well as for the known-variance case in
general, the corresponding results are exact in finite samples, making
a numerical study unnecessary.) We use $1-\alpha = 0.95$ and set to $n
= 40$, $k = 35$, $\xi_{i,n} = 1$, and $\sigma^2 = 1$ (just like in
Figures~\ref{plotpdfH}-\ref{plotpdfAS}). Two different values for the
threshold were chosen, $\eta_{i,n}=0.05$ and $\eta_{i,n}=0.5$. For
each of the two estimators, we use the corresponding lower bound from
Proposition~\ref{confunknownHfs} and \ref{confunknownASfs},
respectively, to compute the length $a_{i,n}$ for a confidence
interval of the form $E_{i,n} = [\tilde\theta_i - \hat\sigma a_{i,n},
  \tilde\theta_i + \hat\sigma a_{i,n}]$ (with $\tilde\theta_i$
denoting either $\tilde\theta_{H,i}$ or $\tilde\theta_{AS,i}$) with
minimal coverage of \emph{at least} 0.95. For the particular length
$a_{i,n}$, we also list the upper bound for the minimal coverage
probability from (\ref{upperbound}). The actual minimal coverage
probability listed below in Table~\ref{sim} is computed by numerically
minimizing the expression $P_{n,\theta,\sigma}(\theta_i \in
[\tilde\theta_i - \hat\sigma a_{i,n}, \tilde\theta_i + \hat\sigma
  a_{i,n}]$) in $\theta_i$ for the given length $a_{i,n}$ (see also
Figure~\ref{pnplots} for a plot of this coverage probability as a
function of $\theta_i$).

\medskip

\begin{table}[htp]
\begin{tabular}{l|cc|c|cc|cc}
& \multicolumn{2}{c|}{length {\small $a_{i,n}$}} & \multirow{2}{*}{lower bound}
& \multicolumn{2}{c|}{actual min.\ coverage} & \multicolumn{2}{c}{upper bound}
\\[0.5ex] 
& {\small $\eta_{i,n}\!=\!0.05$} & {\small $\eta_{i,n}\!=\!0.5$} & & 
{\small $\eta_{i,n}\!=\!0.05$} & {\small $\eta_{i,n}\!=\!0.5$} & 
{\small $\eta_{i,n}\!=\!0.05$} & {\small $\eta_{i,n}\!=\!0.5$} \\[0.5ex] 
\hline 
\hline 
\down{$\hat\theta_{LS,i}$} & \multicolumn{2}{c|}{\down{0.406}} & \down{--} & 
\multicolumn{2}{c|}{\down{0.95}} & \multicolumn{2}{c}{\down{--}} \\[1ex] 
\hline 
\down{$\tilde\theta_{H,i}$} & \down{0.434} & \down{0.823} & \down{0.95} & 
\down{0.9592} & \down{0.9893} & \down{0.9595} & \down{0.9965} \\[1ex] 
\hline 
\down{$\tilde\theta_{AS,i}$} & \down{0.432} & \down{0.820} & \down{0.95} & 
\down{0.9574} & \down{0.9844} & \down{0.9591} & \down{0.9965} \\[1ex]
\hline
\end{tabular}
\caption{\label{sim} Results for the 95\%-confidence intervals based
  on the lower bounds. The parameters were set to $n=40$, $k=35$,
  $\xi_{i,n} = 1$, and $\sigma^2 = 1$. The threshold was
  $\eta_{i,n}=0.05$ and $\eta_{i,n}=0.5$.}
\end{table}
We see that the threshold $\eta_{i,n}$ plays an important role: a
larger threshold yields larger confidence intervals and also widens
the range between lower and upper bound with the actual minimal
coverage probabilities being closer to the upper than the lower
bound. In order to get an idea of the ``conservativeness'' introduced
by using the upper bounds we found that for the larger threshold
$\eta_{i,n}=0.5$, the length $a_{i,n}=0.67$ approximately gave the
desired minimal coverage probability of $0.95$ for
hard-thresholding. This was computed by decreasing the conservative
value of $a_{i,n} = 0.823$ until the numerically computed minimum was
at roughly at $0.95$. Table~\ref{sim} also shows that (for the values
used) the intervals based hard-thresholding are slightly larger than
the ones based on adaptive soft-thresholding, which is in line with
the theoretical findings for the known-variance case.

\medskip

\begin{figure}[htb]
\begin{center}
\includegraphics[width=6cm]{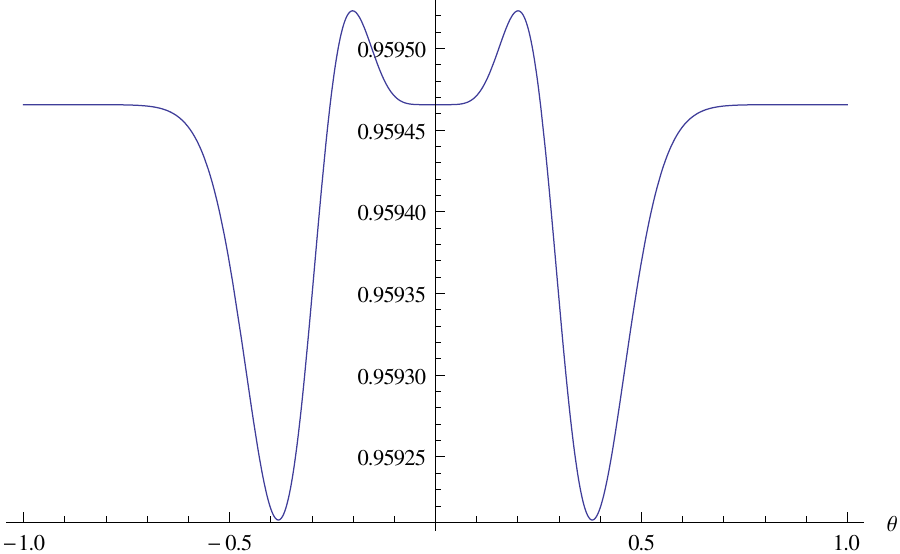} $\;\;\;$
\includegraphics[width=6cm]{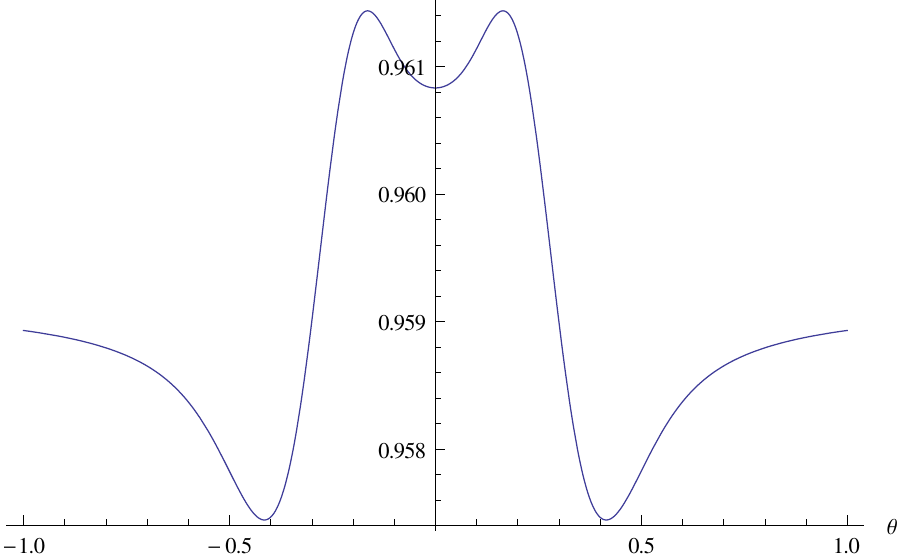}
\caption{\label{pnplots} Plots of coverage probabilities
  $P_{n,\theta,\sigma}(\theta_i \in [\tilde\theta_i - \hat\sigma
    a_{i,n}, \tilde\theta_i + \hat\sigma a_{i,n}]$ as a function of
  $\theta_i$ for hard-thresholding (left) and adaptive
  soft-thresholding (right), where $a_{i,n}$ was chosen for each
  estimator so that the upper bound in
  Proposition~\ref{confunknownHfs} and \ref{confunknownASfs} equals
  $1-\alpha=0.95$. For both plots: $n=40$, $k=35$, $\eta_{i,n}=0.05$,
  $\xi_{i,n}=1$, $\sigma^2=1$.}
\end{center}
\end{figure}

\vspace{-0.3cm}

\section{Summary}
\label{summ}

We considered distributional properties of thresholding estimators
within a normal linear regression model with a potentially diverging
number of parameters where the error variance may be unknown. Aside
from looking at finite-sample as well as large-sample distributions
(within a moving-parameter asymptotic framework) for the estimators,
we constructed valid confidence sets based on these estimators. We
provided intervals that exactly have the prescribed coverage as their
minimal coverage probability in finite samples for the case when the
error variance is known, and conservative intervals where the minimal
coverage probability can never fall below the prescribed value for the
case of unknown error variance. Some numerical considerations were
carried out to illustrate how conservative these lower bounds are and
what the approximation entails for the lengths of the so-constructed
intervals.

We showed that the lengths of the intervals based on the thresholding
estimators are always larger compared to the corresponding standard
intervals in finite samples. Asymptotically, in case of conservative
variable selection, the lengths of all intervals including the
standard one are of the same order. When the estimators are tuned to
perform consistent variable selection, the lengths of the intervals
based on the thresholding estimators are in fact larger by an order of
magnitude compared to the standard one.

\section{Acknowledgments}
\label{ack}

The author wishes to thank anonymous referees for very valuable
comments and acknowledges support from DFG grant FOR 916.


\newpage

\begin{appendix}
\section{Overview of Assumptions and Results}
\label{table}

A quick overview of the propositions and theorems derived in this
paper is given below in Table~\ref{overview}. The first chart lists
all finite-sample results. All of them except
Theorem~\ref{confknownfs} concern the estimators for the case of
unknown error variance. The second chart depicts asymptotic results
under a conservatively tuned regime for the two cases of the
asymptotic behavior of the number of degrees of freedom, namely, when
$n-k$ eventually converges to a constant and when $n-k$ diverges. All
results in this chart are for unknown error variance. Finally, the
asymptotic results under a consistent variable selection framework are
specified in the last chart. Proposition~\ref{confknowncrude} concerns
the known-variance case (for which the asymptotic behavior of $n-k$ is
irrelevant), whereas all other results listed are for the
unknown-variance case. Here, we distinguish three situations for the
large-sample behavior of $n-k$: $n-k$ converging to a finite value,
$n-k$ tending to infinity, and the subcase where $n-k$ diverges ``fast
enough'' in relation to the tuning parameter $\eta_{i,n}$ in the sense
that $n^{1/2}\eta_{i,n}/\sqrt{n-k} \to 0$. To contrast the main
contributions of this paper concerning the confidence intervals with the
auxiliary results on the distributions of the estimators, the former
ones are listed in boldface.

\medskip

\begin{table}[htp]

\begin{tabular}{|c|c|}
\hline
\multicolumn{2}{|c|}{\ddown{\bf finite-sample results}} \\[2ex]
\hline
\hline
\down{distributions} & \down{Prop. \ref{fsdistH}, \ref{fsdistS}, 
\ref{fsdistAS}} \\[1ex]
\hline
\multirow{2}{*}{\ddown{confidence intervals}} & 
\ddown{{\bf Ther.\ \ref{confknownfs}} (known variance)} \\[1ex]
&  \down{\bf  Prop.\ \ref{confunknownHfs}, \ref{confunknownSfs}, 
\ref{confunknownASfs}} \\[1ex]
\hline
\end{tabular}

\bigskip
\medskip

\begin{tabular}{|c|c|c|}
\hline
\multicolumn{3}{|c|}{\ddown{\bf asymptotic results: conservative tuning}}
\\[2ex]
\hline
& \down{$n-k = m$} & \down{$n-k \to \infty$} \\[1ex]
\hline
\hline
\down{distributions} & \down{Prop.\ \ref{lsdistHconserv}, \ref{lsdistSconserv},
\ref{lsdistASconserv}} & 
\down{--} \\[1ex]
\hline
\down{confidence intervals} & \down{--} & \down{\bf Ther.\ 
\ref{confequivH}, \ref{confequivS}, \ref{confequivAS}} \\[1ex]
\hline
\end{tabular}

\bigskip
\medskip

\begin{tabular}{|c|c|c|c|}
\hline
\multicolumn{4}{|c|}{\ddown{\bf asymptotic results: consistent tuning}} 
\\[2ex]
\hline
 & \multirow{2}{*}{\ddown{$n-k = m$}} & 
\multirow{2}{*}{\ddown{$n-k \to \infty$}} & 
\ddown{$\;\;  n^{1/2}\eta_{i,n}/\sqrt{n-k} \to 0 \;\;$} \\[1ex]
& & & \down{($n-k \to \infty$ ``fast enough'')} \\[1ex] 
\hline
\hline 
\down{distributions} & \down{Prop.\ \ref{lsdistHconsist}(a), 
\ref{lsdistSconsist}(a), \ref{lsdistASconsist}(a)} & 
\multicolumn{2}{c|}{\down{Prop.\ \ref{lsdistHconsist}(b), 
\ref{lsdistSconsist}(b), \ref{lsdistASconsist}(b)}} \\[1ex]
\hline
\multirow{2}{*}{\down{confidence intervals}} 
& \multicolumn{2}{c|}{} & \down{\bf Ther.\ 
\ref{confequivH}, \ref{confequivS}, \ref{confequivAS}} \\[1ex]
\cline{2-4}
& \multicolumn{3}{c|}{\down{{\bf Prop.\ \ref{confknowncrude}} 
(known variance), {\bf Prop.\ \ref{confunknowncrude}}}} \\[1ex]
\hline
\end{tabular}

\smallskip

\caption{\label{overview} Overview of assumptions and results.}
\end{table}

\section{Proofs}
\label{proofs}

\subsection{Proofs for Section~\ref{fsdist}}

\begin{proof}[Proof of Proposition~\ref{fsdistH}]
In the following, we denote the distribution function of
$\sigma^{-1}\alpha_{i,n} (\hat\theta_{H,i} - \theta_i)$, the estimator in
the known-variance case, by $H^i_{H,\eta_{i,n},n,\theta,\sigma}$ which was
already derived in \cite{poesch11}. We have
\begin{eqnarray*}
\tilde H_{H,n,\theta ,\sigma }^i(x) &=&
\int_0^\infty P_{n,\theta,\sigma} 
\left(\sigma^{-1}\alpha_{i,n}(\tilde\theta_{H,i}-\theta_i) \leq
x \hat\sigma/\sigma \, \middle| \,
\hat\sigma/\sigma = s\right) \rho_{n-k}(s) \, ds \\
 &=& \int_0^\infty H^i_{H,s\eta_{i,n},n,\theta,\sigma}(xs)\,\rho_{n-k}(s)\,ds,
\end{eqnarray*}
where we have used independence of $\hat\sigma$ and
$\hat\theta_{LS,i}$ allowing us to replace $\hat{\sigma}$ by $s\sigma$
in the relevant formulas, cf.~\citet[p.~110]{leepoe03}. Replacing
$\eta_{i,n}$ by $s\eta_{i,n}$ and $x$ by $xs$ in (7) in
\cite{poesch11} into the above equation gives (\ref{fscdfH}). To show
(\ref{fspdfH0}) and (\ref{fspdfHnot0}), observe first that for each
fixed $s$ we can regard $H^i_{H,s\eta_{i,n},n,\theta,\sigma}(xs)$ as
$H^i_{H,s\eta_{i,n},n,\theta,\sigma}(x)$, but with the scaling factor
$\alpha_{i,n}$ replaced by $\alpha_{i,n}/s$. Further rewriting this
cdf as an integral of the measure
$dH^i_{H,s\eta_{i,n},n,\theta,\sigma}$ from (8) in \cite{poesch11},
applying Fubini's theorem and performing an elementary calculation
yields
\begin{align*}
&\tilde H^i_{H,n,\theta,\sigma}(x) \; = \; \int_0^\infty 
\ind\{xs/\alpha_{i,n} + \theta_i/\sigma \geq 0\} \\
& \times 
\left[ \Phi\big(n^{1/2} (-\theta_i/(\sigma \xi_{i,n}) + s\eta_{i,n})\big)
- \Phi\big(n^{1/2} (-\theta_i/(\sigma \xi_{i,n}) - s\eta_{i,n})\big) \right]
\rho_{n-k}(s) \, ds \\
& + \int_{-\infty}^x \int_0^\infty (n^{1/2}s/(\alpha_{i,n}\xi_{i,n})) 
\phi\big(n^{1/2}ws/(\alpha_{i,n}\xi_{i,n})\big)
\ind\left(|ws/\alpha_{i,n} + \theta_i/\sigma| > \xi_{i,n}s\eta_{i,n}\right) \\
& \times \rho_{n-k}(s) \, ds dw
\end{align*}
for any $\theta_i \in \R$. The second term in the above display
immediately gives the second term in (\ref{fspdfH0}) and
(\ref{fspdfHnot0}). Also, the first term in (\ref{fspdfH0}) can be
derived in a straightforward manner since the indicator function in
the first integral of the above display simply becomes $\ind\{x \geq
0\}$ for $\theta_i = 0$. Finally, to obtain the first term in
(\ref{fspdfHnot0}), ``move'' the corresponding indicator function in
the above display into the limits of the integral and differentiate
with respect to to $x$ applying Leibniz's integral rule.
\end{proof}

\smallskip

\begin{proof}[Proof of Proposition~\ref{fsdistS}]
The proof proceeds in a similar manner as in the proof of
Proposition~\ref{fsdistH}. Let $H^i_{S,\eta_{i,n},n,\theta,\sigma}$
denote the distribution function of $\sigma^{-1}\alpha_{i,n}
(\hat\theta_{S,i} - \theta_i)$, the estimator in the known-variance
case, which was derived in \cite{poesch11}. We have
\begin{eqnarray*}
\tilde H_{S,n,\theta ,\sigma }^i(x) &=&
\int_0^\infty P_{n,\theta,\sigma} 
\left(\sigma^{-1}\alpha_{i,n}(\tilde\theta_{S,i}-\theta_i) \leq
x \hat\sigma/\sigma  \, \middle| \,
\hat\sigma/\sigma = s \right) \rho_{n-k}(s) \, ds \\
 &=& \int_0^\infty H^i_{S,s\eta_{i,n},n,\theta,\sigma}(xs)\,\rho_{n-k}(s)\,ds,
\end{eqnarray*}
where we have used independence of $\hat\sigma$ and
$\hat\theta_{LS,i}$ allowing us to replace $\hat{\sigma}$ by $s\sigma$
in the relevant formulas, cf.~\citet[p.~110]{leepoe03}. Replacing
$\eta_{i,n}$ by $s\eta_{i,n}$ and $x$ by $xs$ in (9) in
\cite{poesch11} into the above equation gives (\ref{fscdfS}). To show
(\ref{fspdfS0}) and (\ref{fspdfSnot0}), observe first that for each
fixed $s$ we can regard $H^i_{S,s\eta_{i,n},n,\theta,\sigma}(xs)$ as
$H^i_{S,s\eta_{i,n},n,\theta,\sigma}(x)$, but with the scaling factor
$\alpha_{i,n}$ replaced by $\alpha_{i,n}/s$. Further rewriting this
cdf as an integral of the measure
$dH^i_{S,s\eta_{i,n},n,\theta,\sigma}$ from (10) in \cite{poesch11},
applying Fubini's theorem and performing an elementary calculation
yields
\begin{align*}
&\tilde H^i_{S,n,\theta,\sigma}(x) \; = \; \int_0^\infty 
\ind\{xs/\alpha_{i,n} + \theta_i/\sigma \geq 0\} \\
& \times 
\left[ \Phi\big(n^{1/2} (-\theta_i/(\sigma \xi_{i,n}) +s\eta_{i,n})\big)
- \Phi\big(n^{1/2} (-\theta_i/(\sigma \xi_{i,n}) - s\eta_{i,n})\big) \right]
\rho_{n-k}(s) \, ds \\
& + \int_{-\infty}^x \int_0^\infty (n^{1/2}s/(\alpha_{i,n}\xi_{i,n}))  
\left[ \phi\big(n^{1/2}(ws/(\alpha_{i,n}\xi_{i,n}) + s\eta_{i,n})\big)
\ind\left(ws/\alpha_{i,n} + \theta_i/\sigma > 0\right)\right. \\
& +\left.\phi\big(n^{1/2}(ws/(\alpha_{i,n}\xi_{i,n}) - s\eta_{i,n})\big)
\ind\left(ws/\alpha_{i,n} + \theta_i/\sigma < 0\right) \right]
\rho_{n-k}(s) \, ds dw
\end{align*}
for any $\theta_i \in \R$. The second term in the above display
immediately gives the second term in (\ref{fspdfS0}) and
(\ref{fspdfSnot0}). Also, the first term in (\ref{fspdfS0}) can be
derived in a straightforward manner since the indicator function in
the first integral of the above display simply becomes $\ind\{x \geq
0\}$ for $\theta_i = 0$. Finally, to obtain the first term in
(\ref{fspdfSnot0}), ``move'' the corresponding indicator function in
the above display into the limits of the integral and differentiate
with respect to to $x$ applying Leibniz's integral rule.
\end{proof}

\smallskip

\begin{proof}[Proof of Proposition~\ref{fsdistAS}]
The proof proceeds in a similar manner as in the proof of
Proposition~\ref{fsdistH}. Let $H^i_{AS,\eta_{i,n},n,\theta,\sigma}$
denote the distribution function of $\sigma^{-1}\alpha_{i,n}
(\hat\theta_{AS,i} - \theta_i)$, the estimator in the known-variance
case, which was derived in \cite{poesch11}. We have
\ref{fsdistS} that
\begin{eqnarray*}
\tilde H_{AS,n,\theta ,\sigma }^i(x) &=&
\int_0^\infty P_{n,\theta,\sigma} 
\left(\sigma^{-1}\alpha_{i,n}(\tilde\theta_{AS,i}-\theta_i) \leq
x \hat\sigma/\sigma \mid
\hat\sigma = s\sigma \right) \rho_{n-k}(s) \, ds \\
 &=& \int_0^\infty H^i_{AS,s\eta_{i,n},n,\theta,\sigma}(xs)\,\rho_{n-k}(s)\,ds,
\end{eqnarray*}
where we have used independence of $\hat\sigma$ and
$\hat\theta_{LS,i}$ allowing us to replace $\hat{\sigma}$ by $s\sigma$
in the relevant formulas, cf.~\citet[p.~110]{leepoe03}. Replacing
$\eta_{i,n}$ by $s\eta_{i,n}$ and $x$ by $xs$ in (11) in
\cite{poesch11} into the above equation gives (\ref{fscdfAS}). To show
(\ref{fspdfAS0}) and (\ref{fspdfASnot0}), observe first that for each
fixed $s$ we can write $H^i_{AS,s\eta_{i,n},n,\theta,\sigma}(xs)$ as
$H^i_{AS,s\eta_{i,n},n,\theta,\sigma}(x)$, but with the scaling factor
$\alpha_{i,n}$ replaced by $\alpha_{i,n}/s$. Further rewriting this
cdf as an integral of the measure
$dH^i_{AS,s\eta_{i,n},n,\theta,\sigma}$ from Proposition~21 in
\cite{poesch11}, applying Fubini's theorem and performing an
elementary calculation yields
\begin{align*}
&\tilde H^i_{AS,n,\theta,\sigma}(x) \; = \; \int_0^\infty 
\ind\{xs/\alpha_{i,n} + \theta_i/\sigma \geq 0\} \\
& \times 
\left[ \Phi\big(n^{1/2} (-\theta_i/(\sigma \xi_{i,n}) + s\eta_{i,n})\big)
- \Phi\big(n^{1/2} (-\theta_i/(\sigma \xi_{i,n}) - s\eta_{i,n})\big) \right]
\rho_{n-k}(s) \, ds \\
& + \int_{-\infty}^x \int_0^\infty 0.5 (n^{1/2}s/(\alpha_{i,n}\xi_{i,n})) 
\left[\phi\big(z_{n,\theta,\sigma}^{(2)}(ws,s\eta_{i,n})\big)
(1 + t_{n,\theta,\sigma}(ws,s\eta_{i,n})) 
\ind\{ws/\alpha_{i,n} + \theta_i/\sigma > 0\} \right. \\
& + \phi\big(z_{n,\theta,\sigma}^{(1)}(ws,s\eta_{i,n})\big)
(1 - t_{n,\theta,\sigma}(ws,s\eta_{i,n})) 
\ind\{ws/\alpha_{i,n} + \theta_i/\sigma < 0\} \left.\right] \rho_{n-k}(s) \, ds \, dw
\end{align*}
for any $\theta_i \in \R$. The second term in the above display
immediately gives the second term in (\ref{fspdfAS0}) and
(\ref{fspdfASnot0}). Also, the first term in (\ref{fspdfAS0}) can also
be derived in a straightforward manner since the indicator function in
the first integral of the above display simply becomes $\ind\{x \geq
0\}$ for $\theta_i = 0$. Finally, to obtain the first term in
(\ref{fspdfASnot0}), ``move'' the corresponding indicator function in
the above display into the limits of the integral and differentiate
with respect to to $x$ applying Leibniz's integral rule.
\end{proof}

\subsection{Proofs for Section~\ref{lsdist}}

\begin{proof}[Proof of Proposition~\ref{lsdistHconserv}]
Note that for fixed $x$ the expression inside the square brackets in
(\ref{fscdfH}) converges to the expression inside the square brackets
in (\ref{lscdfHnspm}) for Lebesgue-almost all $s \in
(0,\infty)$. Since $n-k = m$ eventually, the dominated convergence
theorem proves (\ref{lscdfHnspm}). To conclude the expressions for the
corresponding measure, note that the limit distribution in
(\ref{lscdfHnspm}) is the same as the finite-sample distribution in
(\ref{fscdfH}) with $n^{1/2}\theta_{i,n}/(\sigma_n\xi_{i,n})$ and
$n^{1/2}\eta_{i,n}$ having settled down to their limiting values
$\nu_i$ and $e_i$, respectively, so that the formulas in
(\ref{fspdfH0}) and (\ref{fspdfHnot0}) can be used.
\end{proof}

\smallskip

\begin{proof}[Proof of Proposition~\ref{lsdistSconserv}]
The proof proceeds in a similar manner as the proof of
Proposition~\ref{lsdistHconserv}. Note that for fixed $x$ the
expression inside the square brackets in (\ref{fscdfS}) converges to
the expression inside the square brackets in (\ref{lscdfSnspm}) for
Lebesgue-almost all $s \in (0,\infty)$. Since $n-k = m$ eventually,
the dominated convergence theorem proves (\ref{lscdfSnspm}). To
conclude the expressions for the corresponding measure, note that the
limit distribution in (\ref{lscdfSnspm}) is the same as the
finite-sample distribution in (\ref{fscdfS}) with
$n^{1/2}\theta_{i,n}/(\sigma_n\xi_{i,n})$ and $n^{1/2}\eta_{i,n}$
having settled down to their limiting values $\nu_i$ and $e_i$,
respectively, so that the formulas in (\ref{fspdfS0}) and
(\ref{fspdfSnot0}) can be used.
\end{proof}

\smallskip

\begin{proof}[Proof of Proposition~\ref{lsdistASconserv}]
The proof proceeds in a similar manner as the proof of
Proposition~\ref{lsdistHconserv}. Note that for fixed $x$ the
expression inside the square brackets in (\ref{fscdfAS}) converges to
the expression inside the square brackets in (\ref{lscdfASnspm}) for
Lebesgue-almost all $s \in (0,\infty)$. Since $n-k = m$ eventually,
the dominated convergence theorem proves (\ref{lscdfASnspm}). To
conclude the expressions for the corresponding measure, note that the
limit distribution in (\ref{lscdfASnspm}) is the same as the
finite-sample distribution in (\ref{fscdfAS}) with
$n^{1/2}\theta_{i,n}/(\sigma_n\xi_{i,n})$ and $n^{1/2}\eta_{i,n}$
having settled down to their limiting values $\nu_i$ and $e_i$,
respectively, so that the formulas in (\ref{fspdfAS0}) and
(\ref{fspdfASnot0}) can be used.
\end{proof}

\smallskip

\begin{proof}[Proof of Proposition~\ref{lsdistHconsist}]
To prove part (a) note that by Proposition~\ref{fsdistH} the
distribution function of
$(\hat\sigma\xi_{i,n}\eta_{i,n})^{-1}(\tilde\theta_{H,i} -
\theta_{i,n})$ is given by
\begin{align}
\label{lsHconsistfs}
\begin{split}
\int_0^\infty \big[ \Phi\big(n^{1/2}\eta_{i,n}xs\big)
\ind\big(&|xs + \zeta_{i,n}| > s\big)
+ \Phi\big(n^{1/2}\eta_{i,n}(s - \zeta_{i,n})\big)
\ind\left(0 \leq xs + \zeta_{i,n} \leq s\right) \\
& + \Phi\big(n^{1/2}\eta_{i,n}(-s - \zeta_{i,n})\big)
\ind\left(-s \leq xs + \zeta_{i,n} < 0\right) \big]
\rho_{n-k}(s) \, ds,
\end{split}
\end{align} 
where $\zeta_{i,n} = \theta_{i,n}/(\sigma_n\xi_{i,n}\eta_{i,n})$ and
$\zeta_{i,n} \to \zeta_i$. Since $n^{1/2}\eta_{i,n} \to \infty$ and
$n-k = m$ eventually, the dominated convergence theorem gives that the
above display converges to
\begin{align}
\label{lsHconsistm}
\begin{split}
\int_0^\infty \big[\ind(x \geq 0) \ind(|xs + \zeta_i| & > s) + 
\ind(s > \zeta_i) \ind(0 \leq xs + \zeta_i \leq s) \\
& + \ind(s < -\zeta_i) \ind(-s \leq xs + \zeta_i < 0)\big] \rho_m(s)\,ds
\end{split}
\end{align}
for all $x \neq 0$ since then the integrand of (\ref{lsHconsistfs})
converges to the integrand of (\ref{lsHconsistm}) for Lebesgue-almost
all $s > 0$. The expression in (\ref{lsHconsistm}) clearly simplifies
to $\ind(x \geq 0)$ when $|\zeta_i| = \infty$ or $\zeta_i = 0$,
proving part 1. For $0 < |\zeta_i| < \infty$, after some elementary
calculations, we can write (\ref{lsHconsistm}) as
$$
\ind(-1 \leq x < 0) \int_{\zeta_i}^{-\zeta_i/x} \rho_m(s) \, ds + \ind(x \geq 0)
$$
when $\zeta_i > 0$, and as
$$
\ind(0 \leq x < 1) \big[\int_{-\zeta_i/x}^\infty \rho_m(s) \, ds 
+ \int_0^{-\zeta_i} \rho_m(s) \, ds \big] + \ind(x \geq 1)
$$
when $\zeta_i < 0$, yielding 2.\ and 3.\ and finishing the proof for part (a).

\medskip

For proving part (b), note that the distribution corresponding to
$\rho_{n-k}$, that is, the distribution of $\hat\sigma/\sigma_n$, now
converges to pointmass at 1 in probability. This implies that by
Slutzky's Theorem the limiting distribution of
$\hat\sigma^{-1}\alpha_{i,n} (\tilde\theta_{H,i} - \theta_{i,n}) =
(\hat\sigma/\sigma_n)^{-1}\,\sigma_n^{-1}\alpha_{i,n}
(\tilde\theta_{H,i} - \theta_{i,n})$ is the same as the one of
$\sigma_n^{-1}\alpha_{i,n} (\tilde\theta_{H,i} - \theta_{i,n})$ which
can be found in Theorem~36(b) in \cite{poesch11}.
\end{proof}

\smallskip

\begin{proof}[Proof of Proposition~\ref{lsdistSconsist}]
To prove part (a) note that the cdf
of $\hat\sigma^{-1}\alpha_{i,n} (\tilde\theta_{S,i} - \theta_{i,n})$
with $\alpha_{i,n} = (\xi_{i,n}\eta_{i,n})^{-1}$ is given by
\begin{equation}
\label{lsSconsistfs}
\int_0^\infty \big[ \Phi\big(n^{1/2}\eta_{i,n}(x+1)s)\big)
\ind\left(xs + \zeta_{i,n} \geq 0\right)
+ \Phi\big(n^{1/2}\eta_{i,n}(x-1)s\big)
\ind\left(xs + \zeta_{i,n} < 0\right) \big]
\rho_{n-k}(s) \, ds,
\end{equation}
again, where $\zeta_{i,n} =
\theta_{i,n}/(\sigma_n\xi_{i,n}\eta_{i,n})$ and $\zeta_{i,n} \to
\zeta_i$. Since $n^{1/2}\eta_{i,n} \to \infty$ and $n-k = m$
eventually, the dominated convergence theorem gives that the above
display converges to
\begin{equation}
\label{lsSconsistm}
\int_0^\infty \big[\ind(x \geq -1) \ind(xs + \zeta_i \geq 0) + 
\ind(x \geq 1) \ind(xs + \zeta_i < 0)\big] \rho_m(s)\,ds,
\end{equation}
for all $x \neq -1$ when $\zeta_i > 0$, for all $x \neq 1$ when
$\zeta_i < 0$, and for all $x \neq 0$ when $\zeta_i = 0$, since then
the integrand of (\ref{lsSconsistfs}) converges to the integrand of
(\ref{lsSconsistm}) for Lebesgue-almost all $s > 0$. The expression in
(\ref{lsSconsistm}) simplifies to $\ind(x \geq -1)$ and $\ind(x \geq
1)$ for $\zeta_i = \infty$ and $\zeta_i = -\infty$, respectively, and
to $\ind(x \geq 0)$ for $\zeta_i = 0$, proving parts 1.\ and 3.  For $0
< |\zeta_i| < \infty$, we can write (\ref{lsSconsistm}) as
$$
\ind(-1 \leq x < 0) \int_0^{-\zeta_i/x} \rho_m(s) \, ds + \ind(x \geq 0)
$$
when $\zeta_i > 0$, and as
$$
\ind(0 \leq x < 1) \int_{-\zeta_i/x}^\infty \rho_m(s) \, ds + \ind(x \geq 1)
$$
when $\zeta_i < 0$, yielding 2.\ and 4.\ and finishing the proof for part (a).

\medskip

For proving part (b), note that the distribution corresponding to
$\rho_{n-k}$, that is, the distribution of $\hat\sigma/\sigma_n$, now
converges to pointmass at 1 in probability. This implies that by
Slutzky's Theorem the limiting distribution of
$\hat\sigma^{-1}\alpha_{i,n} (\tilde\theta_{S,i} - \theta_{i,n}) =
(\hat\sigma/\sigma_n)^{-1}\,\sigma_n^{-1}\alpha_{i,n}
(\tilde\theta_{S,i} - \theta_{i,n})$ is the same as the one of
$\sigma_n^{-1}\alpha_{i,n} (\tilde\theta_{S,i} - \theta_{i,n})$ which
can be found in Theorem~37(b) in \cite{poesch11}.
\end{proof}

\smallskip

\begin{proof}[Proof of Proposition~\ref{lsdistASconsist}]
To prove part (a) note that the cdf
of $\hat\sigma^{-1}\alpha_{i,n} (\tilde\theta_{AS,i} - \theta_{i,n})$
with $\alpha_{i,n} = (\xi_{i,n}\eta_{i,n})^{-1}$ can be written as
\begin{align}
\label{lsASconsistfs}
\begin{split}
\int_0^\infty &\big[ \Phi\big(n^{1/2}\eta_{i,n}
\big((xs-\zeta_{i,n})/2 + \sqrt{((xs+\zeta_{i,n})/2)^2 + s^2}\big)\big)
\ind\left(xs + \zeta_{i,n} \geq 0\right) \\
& + \; \Phi\big(n^{1/2}\eta_{i,n}
\big((xs-\zeta_{i,n})/2 - \sqrt{((xs+\zeta_{i,n})/2)^2 + s^2}\big)\big)
\ind\left(xs + \zeta_{i,n} < 0\right) \big] \rho_{n-k}(s) \, ds,
\end{split}
\end{align}
again, where $\zeta_{i,n} =
\theta_{i,n}/(\sigma_n\xi_{i,n}\eta_{i,n})$ and $\zeta_{i,n} \to
\zeta_i$. Since $n^{1/2}\eta_{i,n} \to \infty$ and $n-k = m$
eventually, the dominated convergence theorem gives that the above
display converges to
\begin{align}
\label{lsASconsistm}
\begin{split}
\int_0^\infty &\big[
\ind\big((xs-\zeta_i)/2 + \sqrt{((xs+\zeta_i)/2)^2 + s^2} \geq 0\big)
\ind(xs + \zeta_i \geq 0) \\
& + \; 
\ind\big((xs-\zeta_i)/2 - \sqrt{((xs+\zeta_i)/2)^2 + s^2} \geq 0\big)
\ind(xs + \zeta_i < 0) \big] \rho_m(s)\,ds,
\end{split}
\end{align}
for all $x$ when $0 < |\zeta_i| < \infty$, and for all $x \neq 0$ when
$\zeta_i = 0$, since then the integrand of (\ref{lsASconsistfs})
converges to the integrand of (\ref{lsASconsistm}) for Lebesgue-almost
all $s > 0$. The expression simplifies to $\ind(x \geq 0)$ for
$\zeta_i =0$. To find the limit when $|\zeta_i| = \infty$, we first
consider the case $\zeta_i = \infty$. For large enough $n$, the
integrand in (\ref{lsASconsistfs}) becomes $\ind((xs-\zeta_{i,n})/2 +
\sqrt{((xs+\zeta_{i,n})/2)^2 + s^2} \geq 0)$ so that we need to
determine the limit of
\begin{align*}
\big((xs&-\zeta_{i,n})/2\big) + \sqrt{\big((xs+\zeta_{i,n})/2\big)^2 + s^2} \\
& = \; \big((xs-\zeta_{i,n})/2\big) + \big((xs+\zeta_{i,n})/2\big) 
\sqrt{1+ \big(2s/(xs+\zeta_{i,n})\big)^2} \\
& = \; \big((xs-\zeta_{i,n})/2\big) + 
\big((xs+\zeta_{i,n})/2\big)\big(1 + O\big((s/(xs+\zeta_{i,n})^2\big)\big) \\
& = \; xs + o(1)
\end{align*}
as $\zeta_{i,n} \to \infty$, where we have used a Taylor-expansion of
$\sqrt{1+z}$ around 0. We can therefore conclude that for all fixed $x
\neq 0$, the integrand of (\ref{lsASconsistfs}) converges for
Lebesgue-almost all $s > 0$ to $\ind(xs \geq 0)$ implying that the
weak limit of (\ref{lsASconsistfs}) is $\ind(x \geq 0)$. The proof for
$\zeta_i = -\infty$ works analogously, finishing part 1.

For $0 < |\zeta_i| < \infty$, a tedious but elementary case-by-case
analysis shows that we can write (\ref{lsASconsistm}) as
$$
\ind(-1 \leq x < 0) \int_{-x\zeta_i}^{-\zeta_i/x} \rho_m(s) \, ds + 
\ind(x \geq 0)
$$
when $\zeta_i > 0$, and as
$$
\ind(0 \leq x < 1) \big[\int_{-\zeta_i/x}^\infty \rho_m(s) \, ds +
\int_0^{-\zeta_ix} \rho_m(s) \, ds\big] 
+ \ind(x \geq 1)
$$
when $\zeta_i < 0$, yielding 2.\ and 3.\ and finishing the proof for part (a).

\medskip

For proving part (b), note that the distribution corresponding to
$\rho_{n-k}$, that is, the distribution of $\hat\sigma/\sigma_n$, now
converges to pointmass at 1 in probability. This implies that by
Slutzky's Theorem the limiting distribution of
$\hat\sigma^{-1}\alpha_{i,n} (\tilde\theta_{AS,i} - \theta_{i,n}) =
(\hat\sigma/\sigma_n)^{-1}\,\sigma_n^{-1}\alpha_{i,n}
(\tilde\theta_{AS,i} - \theta_{i,n})$ is the same as the one of
$\sigma_n^{-1}\alpha_{i,n} (\tilde\theta_{AS,i} - \theta_{i,n})$ which
can be found in Theorem~38(b) in \cite{poesch11}.
\end{proof}

\subsection{Proofs for Section~\ref{confs}}

\begin{proof}[Proof of Theorem~\ref{confknownfs}]
We first consider the hard-thresholding estimator. Observe that 
$$
\hat\theta_{H,i}/(\sigma\xi_{i,n}) = 
\left(\hat\theta_{LS,i}/(\sigma\xi_{i,n})\right) 
\ind\left(|\hat\theta_{LS,i}/(\sigma\xi_{i,n})| > \eta_{i,n}\right) 
$$ 
and that $\hat\theta_{LS,i}/(\sigma\xi_{i,n})$ is
$N(\theta_i/(\sigma\xi_{i,n}),1/n)$. \cite{poesch10} derive confidence
intervals for a hard-thresholding estimator for a Gaussian linear
regression model with orthogonal regressors and known error
variance. Identifying $\hat\theta_{LS,i}/(\sigma\xi_{i,n})$ and
$\theta_i/(\sigma \xi_{i,n})$ with $\bar y$ and $\theta$ and making
use of Proposition~2 and Theorem~5 in the above reference by noting that
\begin{align*}
P_{n,\theta,\sigma}&\left(\theta_i \in [\hat\theta_{H,i} - \sigma
  a_{i,n},\hat\theta_{H,i} + \sigma b_{i,n}]\right) \\
& = \; P_{n,\theta,\sigma}\left(\theta_i/(\sigma\xi_{i,n}) \in
\left[\hat\theta_{H,i}/(\sigma\xi_{i,n}) -
  a_{i,n}/\xi_{i,n},\hat\theta_{H,i}/(\sigma\xi_{i,n}) + 
  b_{i,n}/\xi_{i,n}\right]\right)
\end{align*}
immediately gives the result in (a) after replacing $\eta_n$ with
$\eta_{i,n}$ and $a_n$ by $a/\xi_{i,n}$. The results for soft- and
adaptive soft-thresholding in (b) and (c), respectively, follow
analogously by making use of Proposition~1 and 3, respectively, as
well as Theorem~5 in the reference mentioned above.
\end{proof}

\smallskip

\begin{proof}[Proof of Proposition~\ref{confknowncrude}]
Note that 
$$
P_{n,\theta,\sigma}(\theta_i \in D_{i,n}) = 
P_{n,\theta,\sigma}
\left(-d \leq (\sigma\xi_{i,n}\eta_{i,n})^{-1}(\hat\theta_i - \theta_i)
\leq d\right).
$$ 
Propositions~27(b), 28(b), and 29(b) in \cite{poesch11} show that
any accumulation point of the limiting distribution of
$(\sigma\xi_{i,n}\eta_{i,n})^{-1}(\hat\theta_i - \theta_i)$ with
respect to weak convergence is a measure concentrated on $[-1,1]$,
which proves the result for $d>1$. If $d<1$, the same propositions
show that we can always find a sequence $\theta^{(n)}$ such that the
distribution of $(\sigma\xi_{i,n}\eta_{i,n})^{-1}(\hat\theta_i -
\theta_{i,n})$ is concentrated on one of the endpoints of the interval
$[-1,1]$ implying
$$
P_{n,\theta,\sigma}
\left(-d \leq (\xi_{i,n}\eta_{i,n})^{-1}(\hat\theta_i - \theta_{i,n})
\leq d\right) \longrightarrow 0
$$ 
and proving the claim for $d < 1$. Finally, use the expressions for the
infimal coverage probabilities in Theorem~\ref{confknownfs} to see
the result for $d=1/2$.
\end{proof}

\smallskip

\begin{proof}[Proof of Proposition~\ref{confunknownHfs}]
Define $p^i_{H,n}(\theta,\sigma,a_{i,n},\eta_{i,n}) =
P_{n,\theta,\sigma}(\theta_i \in C_{H,i,n})$, the coverage probability
for the hard-thresholding estimator with known error variance. As
discussed in the beginning of Section~\ref{confunknown}, we have
\begin{align*}
\inf_{\theta \in \R,\sigma \in \R_+} & 
P_{n,\theta,\sigma}(\theta \in E_{H,i,n}) = 
\inf_{\theta \in \R} \int_0^\infty p^i_{H,n}(\theta,1,a_{i,n}s,\eta_{i,n}s)
\rho_{n-k}(s)\,ds \\
& \geq \; 
\int_0^\infty \inf_{\theta \in \R} p^i_{H,n}(\theta,1,a_{i,n}s,\eta_{i,n}s)
\rho_{n-k}(s)\,ds
\end{align*}
which equals 
\begin{align*}
\int_0^\infty \big[\Phi(n^{1/2}&(a_{i,n}/\xi_{i,n} - \eta_{i,n})s) -
  \Phi(-n^{1/2}a_{i,n}s/\xi_{i,n})\big] \rho_{n-k}(s)\,ds \\
& = \;T_{n-k}(n^{1/2}(a_{i,n}/\xi_{i,n}-\eta_{i,n})) - 
T_{n-k}(-n^{1/2}a_{i,n}/\xi_{i,n})
\end{align*}
when $a_{i,n} \geq \xi_{i,n}\eta_{i,n}/2$, and 0 for $a_{i,n} <
\xi_{i,n}\eta_{i,n}/2$, cf. Theorem~\ref{confknownfs}(a).
\end{proof}

\smallskip

\begin{proof}[Proof of Theorem~\ref{confequivH}]
For $p^i_{H,n}(\theta,\sigma,a_{i,n},\eta_{i,n}) =
p^i_{H,n}(\theta/\sigma,1,a_{i,n},\eta_{i,n}) =
P_{n,\theta,\sigma}(\theta_i \in C_{H,i,n})$, the coverage probability
in the known-variance case, we have 
$$
\lim_{\theta_i \to \infty} p^i_{H,n}(\theta,1,a_{i,n},\eta_{i,n})
= \Phi(n^{1/2}a_{i,n}/\xi_{i,n}) - \Phi(-n^{1/2}a_{i,n}/\xi_{i,n}),
$$
as can be seen from Proposition~19 in \cite{poesch11}. This implies that
\begin{equation}
\label{confknownHub}
\inf_{\theta \in \R^k} P_{n,\theta,\sigma}(\theta_i \in C_{H,i,n})
\leq  \Phi(n^{1/2}a_{i,n}/\xi_{i,n}) - \Phi(-n^{1/2}a_{i,n}/\xi_{i,n}),
\end{equation}
as well as
\begin{align}
\notag
\inf_{\theta \in \R^k,\sigma \in \R_+} &
P_{n,\theta,\sigma}(\theta_i \in E_{H,i,n}) \leq 
\lim_{\theta_i \to \infty} \int_0^{\infty} 
p^i_{H,n}(\theta,1,a_{i,n}s,\eta_{i,n}s) \rho_{n-k}(s)\,ds \\
\notag
& = \; \int_0^{\infty} \left[\Phi(n^{1/2}a_{i,n}s/\xi_{i,n}) - 
\Phi(-n^{1/2}a_{i,n}s/\xi_{i,n})\right] \rho_{n-k}(s)\,ds \\
\label{confunknownHub}
& = \; T_{n-k}(n^{1/2}a_{i,n}/\xi_{i,n}) - T_{n-k}(-n^{1/2}a_{i,n}/\xi_{i,n}),
\end{align}
where we have used dominated convergence for the first equality in the
above display.

Step 1: If $n^{1/2}a_{i,n}/\xi_{i,n} \to 0$, the upper bounds in
(\ref{confknownHub}) and (\ref{confunknownHub}) both converge to 0,
thus proving the claim.

Step 2: Let $n^{1/2}\eta_{i,n} \to 0$. If $a_{i,n} <
\xi_{i,n}\eta_{i,n}/2$, then $n^{1/2}a_{i,n}/\xi_{i,n} \to 0$ also,
showing the theorem by Step 1. If $a_{i,n} \geq
\xi_{i,n}\eta_{i,n}/2$, we have by Theorem~\ref{confknownfs}(a) that
$\inf_{\theta \in \R^k}P_{n,\theta,\sigma}(\theta_i \in C_{H,i,n})$
differs from $\Phi(n^{1/2}a_{i,n}/\xi_{i,n}) -
\Phi(-n^{1/2}a_{i,n}/\xi_{i,n})$ only by a term that is $o(1)$ since
$\Phi$ is globally Lipschitz. The same is true for $\inf_{\theta \in
  \R^k,\sigma \in \R_+}P_{n,\theta,\sigma}(\theta_i \in E_{H,i,n})$
since then the difference between lower bound from
Proposition~\ref{confunknownHfs} and the upper bound from
(\ref{confunknownHub}) converges to zero, so that $|\inf_{\theta \in
  \R^k,\sigma \in \R_+}P_{n,\theta,\sigma}(\theta_i \in E_{H,i,n}) -
[\Phi(n^{1/2}a_{i,n}/\xi_{i,n}) - \Phi(-n^{1/2}a_{i,n}/\xi_{i,n})]| \leq
o(1) + 2\|\Phi - T_{n-k}\|_\infty \to 0$ by Polya's Theorem.

Step 3: Assume $n^{1/2}(a_{i,n}/\xi_{i,n} - \eta_{i,n}) \to
\infty$. We then have $n^{1/2}a_{i,n}/\xi_{i,n} \to \infty$ also, so
that by Theorem~\ref{confknownfs}(a) together with
Proposition~\ref{confunknownHfs} the infimal coverage probabilities
both converge to 1, showing the claim for this step.

Step 4: By a subsequence argument we may now assume that
$n^{1/2}a_{i,n}/\xi_{i,n}$ as well as $n^{1/2}\eta_{i,n}$ are bounded
away from zero, and that $n^{1/2}(a_{i,n}/\xi_{i,n}-\eta_{i,n})$ is
bounded from above. Note that this implies that
$a_{i,n}/(\xi_{i,n}\eta_{i,n})$ is also bounded from above by some
constant, say, $C$. By Lemma 13 in \cite{poesch10} [after identifying
  $n$ with $n-k$ and $h_n$ with $\rho_{n-k}$] we see that for every
$\eps > 0$ we can find a constant $c(\eps)>0$ such that for every $n
\geq 2$
\begin{equation}
\label{lemma}
\int_{1-c(\eps)(n-k)^{-1/2}}^{1+c(\eps)(n-k)^{-1/2}} \rho_{n-k}(s)\,ds \geq 1 -
 \eps.
\end{equation}
Now define $\theta^{(n)}(\eps)$ to have $i$-th component
$\theta_{i,n}(\eps) := a_{i,n}(1 + c(\eps)(n-k)^{-1/2})$ and set the
remaining components to arbitrary values. By Proposition~19 in
\cite{poesch11}, this choice of $\theta^{(n)}$ implies that for large
enough $n$ we have
\begin{align*}
&|p^i_{H,n}(\theta^{(n)}(\eps),1,a_{i,n},\eta_{i,n}) - 
\inf_{\theta \in \R^k}p^i_{H,n}(\theta,1,a_{i,n},\eta_{i,n})| \\ 
& = \; |\max(0,\Phi(n^{1/2}(\theta_{i,n}(\eps)/\xi_{i,n}-\eta_{i,n})) -
\Phi(-n^{1/2}a_{i,n}/\xi_{i,n})) \\
& \;\;\;\;\;\; -\max(0,\Phi(n^{1/2}(a_{i,n}/\xi_{i,n}-\eta_{i,n})) - 
\Phi(-n^{1/2}a_{i,n}/\xi_{i,n}))| \\
& \leq \; (2\pi)^{-1/2}n^{1/2}|\theta_{i,n}(\eps)/\xi_{i,n} - 
a_{i,n}/\xi_{i,n}| 
= (2\pi)^{-1/2} c(\eps)(n-k)^{-1/2} n^{1/2}a_{i,n}/\xi_{i,n} \\
& \leq \; (2\pi)^{-1/2}c(\eps)C n^{1/2}\eta_{i,n}/(n-k)^{1/2}\longrightarrow 0,
\end{align*}
where we have used the fact that $\max(0,x)$ is globally Lipschitz
with constant 1 and $\Phi(x)$ is globally Lipschitz with constant
$(2\pi)^{-1/2}$.  Moreover, for $s$ satisfying $|s-1| <
c(\eps)(n-k)^{-1/2}$ we have by the boundedness of
$a_{i,n}/(\xi_{i,n}\eta_{i,n})$ that $sa_{i,n} < \theta_{i,n}(\eps)
\leq s(a_{i,n} + \xi_{i,n}\eta_{i,n})$ for large enough $n$. For such
$s$ and $n$, this implies, in a similar manner as above,
\begin{align*}
|&p^i_{H,n}(\theta^{(n)}(\eps),1,a_{i,n}s,\eta_{i,n}s) - 
p^i_{H,n}(\theta^{(n)}(\eps),1,a_{i,n},\eta_{i,n})| \\
& = \; |\max(0,\Phi(n^{1/2}(\theta_{i,n}(\eps)/\xi_{i,n}-\eta_{i,n}s)) -
\Phi(-n^{1/2}a_{i,n}s/\xi_{i,n})) \\
& \;\;\;\;\;\; - 
\max(0,\Phi(n^{1/2}(\theta_{i,n}(\eps)/\xi_{i,n}-\eta_{i,n})) -
\Phi(-n^{1/2}a_{i,n}/\xi_{i,n}))| \\
& \leq \; (2\pi)^{-1/2}|s-1|n^{1/2}(a_{i,n}/\xi_{i,n} + \eta_{i,n})\\
& \leq \; (2\pi)^{-1/2}c(\eps)(1+C)n^{1/2}\eta_{i,n}/(n-k)^{1/2} 
\longrightarrow 0.
\end{align*}
This furthermore entails that 
\begin{align*}
& \inf_{\theta \in \R,\sigma \in \R_+} 
P_{n,\theta,\sigma}(\theta_i \in E_{H,i,n}) \; \leq \; 
\int_0^\infty p^i_{H,n}(\theta_n(\eps),1,a_{i,n}s,\eta_{i,n}s)
\rho_{n-k}(s)\,ds\\ 
& \leq \; 
\int_0^\infty \left|p^i_{H,n}(\theta_n(\eps),1,a_{i,n}s,\eta_{i,n}s) -
p^i_{H,n}(\theta^{(n)}(\eps),1,a_{i,n},\eta_{i,n})
\right|\rho_{n-k}(s)\,ds \\
& \;\;\;\;\;\;\;\;\; + p^i_{H,n}(\theta^{(n)}(\eps),1,a_{i,n},\eta_{i,n}) \\ 
& \leq \; 
\int_{1-c(\eps)(n-k)^{-1/2}}^{1+c(\eps)(n-k)^{-1/2}}
\left|p^i_{H,n}(\theta_n(\eps),1,a_{i,n}s,\eta_{i,n}s) -
p^i_{H,n}(\theta^{(n)}(\eps),1,a_{i,n},\eta_{i,n})\right|\rho_{n-k}(s)\,ds \\
& \;\;\;\;\;\;\;\;\; +\eps+p^i_{H,n}(\theta^{(n)}(\eps),1,a_{i,n},\eta_{i,n}) \\
& = \; \inf_{\theta \in \R^k}p^i_{H,n}(\theta,1,a_{i,n},\eta_{i,n}) 
+ \eps + o(1)
= \; \inf_{\theta \in \R^k}P_{n,\theta,\sigma}(\theta_i \in C_{H,i,n}) + \eps
+ o(1), 
\end{align*}
where we have made use of (\ref{lemma}) and $\eps > 0$ was arbitrary. On
the other hand, Proposition~\ref{confunknownHfs} and
Theorem~\ref{confknownfs}(a) show that
\begin{align*}
\inf_{\theta \in \R^k,\sigma \in \R_+} &
P_{n,\theta,\sigma}(\theta_i \in E_{H,i,n}) \geq
\max[0,T_{n-k}(n^{1/2}(a_{i,n}/\xi_{i,n} - \eta_{i,n})) -
  T_{n-k}(-n^{1/2}a_{i,n}/\xi_{i,n})] \\
& \geq \; \max[0,\Phi(n^{1/2}(a_{i,n}/\xi_{i,n} - \eta_{i,n})) -
  \Phi(-n^{1/2}a_{i,n}/\xi_{i,n})] - 2\|\Phi - T_{n-k}\|_\infty \\
& = \; \inf_{\theta \in \R^k}
P_{n,\theta,\sigma}(\theta_i \in C_{H,i,n}) - 2\|\Phi - T_{n-k}\|_\infty, 
\end{align*}
where $\|\Phi - T_{n-k}\|_\infty \to 0$ when $n-k
\to \infty$ by Polya's Theorem, finally proving the claim.
\end{proof}

\smallskip

\begin{proof}[Proof of Proposition~\ref{confunknownSfs}]
Define $p^i_{S,n}(\theta,\sigma,a_{i,n},\eta_{i,n}) =
P_{n,\theta,\sigma}(\theta_i \in [\hat\theta_{S,i} - \sigma
  a_{i,n},\hat\theta_{S,i} - \sigma a_{i,n}])$, the coverage
probability for the soft-thresholding estimator with known error
variance. We have
\begin{align*}
& T_{n-k}(n^{1/2}(a_{i,n}/\xi_{i,n}-\eta_{i,n})) - 
T_{n-k}(n^{1/2}(-a_{i,n}/\xi_{i,n}-\eta_{i,n})) \\
& = \; \int_0^\infty \big[\Phi(n^{1/2}(a_{i,n}/\xi_{i,n} - \eta_{i,n})s) -
  \Phi(n^{1/2}(-a_{i,n}/\xi_{i,n}-\eta_{i,n})s)\big] \rho_{n-k}(s)\,ds \\
& = \int_0^\infty \lim_{\theta_i \to \infty} 
p^i_{S,n}(\theta,1,a_{i,n}s,\eta_{i,n}s) \rho_{n-k}(s)\,ds 
= \lim_{\theta_i \to \infty}  \int_0^\infty 
p^i_{S,n}(\theta,1,a_{i,n}s,\eta_{i,n}s) \rho_{n-k}(s)\,ds \\
& \geq \; \inf_{\theta \in \R^k} P_{n,\theta,1}(\theta_i \in E_{S,i,n})
\; \geq \; \int_0^\infty \inf_{\theta \in \R^k}  
p^i_{S,n}(\theta,1,a_{i,n}s,\eta_{i,n}s) \rho_{n-k}(s)\,ds \\
& = \; \int_0^\infty \big[\Phi(n^{1/2}(a_{i,n}/\xi_{i,n} - \eta_{i,n})s) -
  \Phi(n^{1/2}(-a_{i,n}/\xi_{i,n}-\eta_{i,n})s)\big] \rho_{n-k}(s)\,ds \\
& = \; T_{n-k}(n^{1/2}(a_{i,n}/\xi_{i,n}-\eta_{i,n})) - 
T_{n-k}(n^{1/2}(-a_{i,n}/\xi_{i,n}-\eta_{i,n})),
\end{align*}
where the second equality can be seen from Proposition~20 in
\cite{poesch11}, the third equality is due to dominated convergence,
and the second-last equality comes from Theorem~\ref{confknownfs}(b).
\end{proof}

\smallskip

\begin{proof}[Proof of Proposition~\ref{confunknownASfs}]
Define $p^i_{AS,n}(\theta,\sigma,a_{i,n},\eta_{i,n}) =
P_{n,\theta,\sigma}(\theta_i \in [\hat\theta_{AS,i} - \sigma
  a_{i,n},\hat\theta_{AS,i} - \sigma a_{i,n}])$, the coverage
probability for the adaptive soft-estimator with known error
variance. As discussed in the beginning of Section~\ref{confunknown},
we have
\begin{align*}
\inf_{\theta \in \R,\sigma \in \R_+} & 
P_{n,\theta,\sigma}(\theta_i \in E_{AS,i,n}) = 
\inf_{\theta \in \R^k} \int_0^\infty p^i_{AS,n}(\theta,1,a_{i,n}s,\eta_{i,n}s)
\rho_{n-k}(s)\,ds \\
& \geq \; 
\int_0^\infty \inf_{\theta \in \R^k} p^i_{AS,n}(\theta,1,a_{i,n}s,\eta_{i,n}s)
\rho_{n-k}(s)\,ds \\
& = \; \int_0^\infty \big[\Phi(n^{1/2}(a_{i,n}/\xi_{i,n} - \eta_{i,n})s) -
\Phi(-n^{1/2}s\sqrt{(a_{i,n}/\xi_{i,n})^2+\eta^2_{i,n}})\big] \rho_{n-k}(s)\,ds
 \\
& = \;T_{n-k}(n^{1/2}(a_{i,n}/\xi_{i,n}-\eta_{i,n})) - 
T_{n-k}(-n^{1/2}\sqrt{(a_{i,n}/\xi_{i,n})^2+\eta^2_{i,n}})
\end{align*}
where the second-last equality holds by Theorem~\ref{confknownfs}(c).
\end{proof}

\smallskip

\begin{proof}[Proof of Theorem~\ref{confequivAS}]
The proof proceeds in a similar manner as the proof of
Theorem~\ref{confequivH}. For
$p^i_{AS,n}(\theta,\sigma,a_{i,n},\eta_{i,n}) =
p^i_{AS,n}(\theta/\sigma,1,a_{i,n},\eta_{i,n}) =
P_{n,\theta,\sigma}(\theta_i \in C_{AS,i,n})$, the coverage
probability in the known-variance case, we have
$$
\lim_{\theta_i \to \infty} p^i_{AS,n}(\theta,1,a_{i,n},\eta_{i,n})
= \Phi(n^{1/2}a_{i,n}/\xi_{i,n}) - \Phi(-n^{1/2}a_{i,n}/\xi_{i,n}),
$$ 
as can be seen from Proposition~21 in \cite{poesch11}. This implies that
\begin{equation}
\label{confknownASub}
\inf_{\theta \in \R^k} P_{n,\theta,\sigma}(\theta_i \in C_{AS,i,n}) \leq 
\Phi(n^{1/2}a_{i,n}/\xi_{i,n}) - \Phi(-n^{1/2}a_{i,n}/\xi_{i,n})
\end{equation}
as well as
\begin{align}
\notag
\inf_{\theta \in \R^k,\sigma \in \R_+} &
P_{n,\theta,\sigma}(\theta_i \in E_{AS,i,n}) \leq 
\lim_{\theta_i \to \infty} \int_0^{\infty} 
p^i_{AS,n}(\theta,1,a_{i,n}s,\eta_{i,n}s) \rho_{n-k}(s)\,ds \\
\notag
& = \; \int_0^{\infty} \left[\Phi(n^{1/2}a_{i,n}s/\xi_{i,n}) - 
\Phi(-n^{1/2}a_{i,n}s/\xi_{i,n})\right] \rho_{n-k}(s)\,ds \\
\label{confunknownASub}
& = \; T_{n-k}(n^{1/2}a_{i,n}/\xi_{i,n}) - T_{n-k}(-n^{1/2}a_{i,n}/\xi_{i,n}),
\end{align}
where we have used dominated convergence for the first equality in the
above display.

Step 1: If $n^{1/2}a_{i,n}/\xi_{i,n} \to 0$, the upper bounds in
(\ref{confknownASub}) and (\ref{confunknownASub}) both converge to 0,
thus proving the claim.

Step 2: Let $n^{1/2}\eta_{i,n} \to 0$. By Theorem~\ref{confknownfs}(c)
$\inf_{\theta \in \R^k}P_{n,\theta,\sigma}(\theta_i \in C_{H,i,n})$
differs from $\Phi(n^{1/2}a_{i,n}/\xi_{i,n}) -
\Phi(-n^{1/2}a_{i,n}/\xi_{i,n})$ only by a term that is $o(1)$ since
$\Phi$ is globally Lipschitz and $0 \leq
n^{1/2}\sqrt{(a_{i,n}/\xi_{i,n})^2 + \eta_{i,n}^2} -
n^{1/2}a_{i,n}/\xi_{i,n} \leq n^{1/2}\eta_{i,n} \to 0$. The same is
true for $\inf_{\theta \in \R^k,\sigma \in
  \R_+}P_{n,\theta,\sigma}(\theta_i \in E_{H,i,n})$ since then the
difference between the lower bound from Proposition~\ref{confunknownASfs}
and the upper bound from (\ref{confunknownASub}) converges to zero by
a similar reasoning, so that $\inf_{\theta \in \R^k,\sigma \in
  \R_+}P_{n,\theta,\sigma}(\theta_i \in E_{H,i,n}) -
[\Phi(n^{1/2}a_{i,n}/\xi_{i,n}) - \Phi(-n^{1/2}a_{i,n}/\xi_{i,n})] =
o(1) + 2\|\Phi - T_{n-k}\|_\infty \to 0$ by Polya's Theorem.

Step 3: Assume $n^{1/2}(a_{i,n}/\xi_{i,n} - \eta_{i,n}) \to
\infty$. We then have $n^{1/2}\sqrt{(a_{i,n}/\xi_{i,n})^2 +
  \eta_{i,n}^2} \to \infty$ also, so that by
Theorem~\ref{confknownfs}(c) together with
Proposition~\ref{confunknownASfs} the infimal coverage probabilities
both converge to 1, showing the claim for this step.

Step 4: By a subsequence argument we may now assume that
$n^{1/2}a_{i,n}/\xi_{i,n}$ as well as $n^{1/2}\eta_{i,n}$ are bounded
away from zero, and that $n^{1/2}(a_{i,n}/\xi_{i,n}-\eta_{i,n})$ is
bounded from above. Note that this implies that
$a_{i,n}/(\xi_{i,n}\eta_{i,n})$ is also bounded from above by some
constant, say, $C$. Again, for an arbitrary $\eps > 0$ define
$\theta^{(n)}(\eps)$ to have $i$-th component $\theta_{i,n}(\eps) :=
a_{i,n}(1 + c(\eps)(n-k)^{-1/2})$, where $c(\eps)$ is the
constant from (\ref{lemma}) and set the remaining components to
arbitrary values. By Proposition~21 in \cite{poesch11}, this choice of
$\theta^{(n)}$ implies that
$p^i_{AS,n}(\theta^{(n)}(\eps),1,a_{i,n},\eta_{i,n})$ is given by
\begin{align*}
\Phi&\left(n^{1/2}((a_{i,n}-\theta_{i,n}(\eps))/(2\xi_{i,n}) 
+ \sqrt{((a_{i,n}+\theta_{i,n}(\eps))/(2\xi_{i,n}))^2 +\eta_{i,n}^2})\right)\\
& - \; \Phi\left(n^{1/2}(-(a_{i,n}+\theta_{i,n}(\eps))/(2\xi_{i,n}) 
+\sqrt{((-a_{i,n}+\theta_{i,n}(\eps))/(2\xi_{i,n}))^2 + \eta_{i,n}^2})\right).
\end{align*}
Now define $\delta_{i,n} = a_{i,n}c(\eps)(n-k)^{-1/2}$. Using
Theorem~\ref{confknownfs}(c), the fact that $\Phi$ is globally
Lipschitz with constant $(2\pi)^{-1/2}$, and the elementary inequality
$|\sqrt{x} - \sqrt{y}| \leq |x-y|/(2\sqrt{z})$ for $\min(x,y) \geq z >
0$ applied for $z=(\xi_{i,n}\eta_{i,n})^2$, some lengthy but
elementary calculations yield
\begin{align*}
&|p^i_{AS,n}(\theta^{(n)}(\eps),1,a_{i,n},\eta_{i,n}) - 
\inf_{\theta \in \R^k}p^i_{AS,n}(\theta,1,a_{i,n},\eta_{i,n})| \\ 
& \leq \; \left|\Phi\left((n^{1/2}/\xi_{i,n})(a_{i,n}+\delta_{i,n}/2
-\sqrt{(\delta_{i,n}/2)^2+\xi_{i,n}^2\eta_{i,n}^2})\right)
- \Phi\left((n^{1/2}/\xi_{i,n})(a_{i,n}-\xi_{i,n}\eta_{i,n})\right) \right| \\
& \;\;\; + \left|\Phi\left((n^{1/2}/\xi_{i,n})(\delta_{i,n}/2
-\sqrt{(a_{i,n} + \delta_{i,n}/2)^2+\xi_{i,n}^2\eta_{i,n}^2})\right) - 
\Phi\left(-(n^{1/2}/\xi_{i,n})\sqrt{a_{i,n}^2 + \xi_{i,n}^2\eta_{i,n}^2}\right)
\right| \\
& \leq \; (2\pi)^{-1/2}(n^{1/2}/\xi_{i,n})\big[
\delta_{i,n} + \left|\sqrt{(\delta_{i,n}/2)^2+\xi_{i,n}^2\eta_{i,n}^2} - 
\xi_{i,n}\eta_{i,n}\right| \\ 
& \;\;\; + \left| \sqrt{(a_{i,n} + \delta_{i,n}/2)^2+\xi_{i,n}^2\eta_{i,n}^2} 
- \sqrt{a_{i,n}^2+\xi_{i,n}^2\eta_{i,n}^2} \right|
\big] \\
& \leq \; (2\pi)^{-1/2}(n^{1/2}/\xi_{i,n}) \big[1.5\delta_{i,n} 
+ (2\xi_{i,n}\eta_{i,n})^{-1}(\delta_{i,n}a_{i,n} + \delta_{i,n}^2/4)\big] \\
& \leq \; 0.5 (2\pi)^{-1/2}(n^{1/2}/\xi_{i,n})\delta_{i,n}
\big[3 + C + 0.25 \delta_{i,n}/(\xi_{i,n}\eta_{i,n}) \big] \\
& \leq \; 0.5 (2\pi)^{-1/2}c(\eps)Cn^{1/2}\eta_{i,n}(n-k)^{-1/2}
\big[3 + C +0.25Cc(\eps)(n-k)^{-1/2}\big]
\longrightarrow 0.
\end{align*}
Moreover, for $s$ satisfying $|s-1| < c(\eps)(n-k)^{-1/2}$ we have
that $sa_{i,n} < \theta^{(n)}(\eps)$ for all $n$. Again, with some
lengthy but elementary calculations using the above mentioned
inequality twice with
$z=(\xi_{i,n}\eta_{i,n}(1-c(\eps)(n-k)^{-1/2})^2$ we have for $n$
large enough satisfying $1-c(\eps)(n-k)^{-1/2} \geq 1/2$ (entailing
$(2\sqrt{z})^{-1} \leq (\xi_{i,n}\eta_{i,n})^{-1}$, as well as
$1+c(\eps)(n-k)^{-1/2} \leq 3/2$ and $|s+1| \leq 5/2$) that
\begin{align*}
|&p^i_{AS,n}(\theta^{(n)}(\eps),1,a_{i,n}s,\eta_{i,n}s) - 
p^i_{AS,n}(\theta^{(n)}(\eps),1,a_{i,n},\eta_{i,n})| \\
& \leq \; (2\pi)^{-1/2} (n^{1/2}/\xi_{i,n}) \bigg[a_{i,n}|s-1| \\
& \;\;\; + 
\left|\sqrt{((a_{i,n}s+\theta_{i,n}(\eps))/2)^2 + \xi_{i,n}^2\eta_{i,n}^2s^2} -
\sqrt{((a_{i,n}+\theta_{i,n}(\eps))/2)^2 + \xi_{i,n}^2\eta_{i,n}^2} \right|\\
& \;\;\; + 
\left|\sqrt{((a_{i,n}s-\theta_{i,n}(\eps))/2)^2 + \xi_{i,n}^2\eta_{i,n}^2s^2} -
\sqrt{((a_{i,n}-\theta_{i,n}(\eps))/2)^2 + \xi_{i,n}^2\eta_{i,n}^2} 
\right|\bigg]\\
& \leq \; (2\pi)^{-1/2} (n^{1/2}/\xi_{i,n})|s-1| \big[a_{i,n}
+ (\xi_{i,n}\eta_{i,n})^{-1}(a_{i,n}^2|s+1|/2 + a_{i,n}\theta_{i,n}(\eps)
+ 2\xi_{i,n}^2\eta_{i,n}^2|s+1|)\big] \\
& \leq \; (2\pi)^{-1/2}c(\eps)n^{1/2}\eta_{i,n}(n-k)^{-1/2}\big[5 + C + 
(11/4)C^2\big] \longrightarrow 0.
\end{align*}
This furthermore implies that
\begin{align*}
& \inf_{\theta \in \R,\sigma \in \R_+} 
P_{n,\theta,\sigma}(\theta_i \in E_{AS,i,n}) \; \leq \; 
\int_0^\infty p^i_{AS,n}(\theta_n(\eps),1,a_{i,n}s,\eta_{i,n}s)\rho_{n-k}(s)\,ds
\\  & \leq \; 
\int_0^\infty \left|p^i_{AS,n}(\theta_n(\eps),1,a_{i,n}s,\eta_{i,n}s) -
p^i_{AS,n}(\theta^{(n)}(\eps),1,a_{i,n},\eta_{i,n})
\right|\rho_{n-k}(s)\,ds \\
& \;\;\;\;\;\;\;\;\; + p^i_{AS,n}(\theta^{(n)}(\eps),1,a_{i,n},\eta_{i,n}) \\ 
& \leq \; 
\int_{1-c(\eps)(n-k)^{-1/2}}^{1+c(\eps)(n-k)^{-1/2}}
\left|p^i_{AS,n}(\theta_n(\eps),1,a_{i,n}s,\eta_{i,n}s) -
p^i_{AS,n}(\theta^{(n)}(\eps),1,a_{i,n},\eta_{i,n})\right|\rho_{n-k}(s)\,ds \\
& \;\;\;\;\;\;\;\;\; +\eps+p^i_{AS,n}(\theta^{(n)}(\eps),1,a_{i,n},\eta_{i,n})\\
& = \; \inf_{\theta \in \R^k}p^i_{AS,n}(\theta,1,a_{i,n},\eta_{i,n}) 
+ \eps + o(1)
= \; \inf_{\theta \in \R}P_{n,\theta,\sigma}(\theta_i \in C_{AS,i,n}) + \eps
+ o(1), 
\end{align*}
where $\eps > 0$ was arbitrary. On the other hand,
Proposition~\ref{confunknownASfs} and Theorem~\ref{confknownfs}(c) show
that
\begin{align*}
\inf_{\theta \in \R^k,\sigma \in \R_+} &
P_{n,\theta,\sigma}(\theta_i \in E_{AS,i,n}) \\
 & \geq \; T_{n-k}(n^{1/2}(a_{i,n}/\xi_{i,n} - \eta_{i,n})) -
  T_{n-k}(-n^{1/2}\sqrt{(a_{i,n}/\xi_{i,n})^2+\eta_{i,n}^2}) \\
& \geq \; \Phi(n^{1/2}(a_{i,n}/\xi_{i,n} - \eta_{i,n})) -
  \Phi(-n^{1/2}\sqrt{(a_{i,n}/\xi_{i,n})^2+\eta_{i,n}^2}) - 
2\|\Phi - T_{n-k}\|_\infty \\
& = \; \inf_{\theta \in \R^k}
P_{n,\theta,\sigma}(\theta_i \in C_{AS,i,n}) - 2\|\Phi - T_{n-k}\|_\infty, 
\end{align*}
where $\|\Phi - T_{n-k}\|_\infty \to 0$ when $n-k
\to \infty$ by Polya's Theorem, finally proving the claim.
\end{proof}

\smallskip

\begin{proof}[Proof of Proposition~\ref{confunknowncrude}]
Note that 
$$
P_{n,\theta,\sigma}(\theta_i \in F_{i,n}) = 
P_{n,\theta,\sigma}
\left(-d \leq (\hat\sigma\xi_{i,n}\eta_{i,n})^{-1}(\tilde\theta_i - \theta_i)
\leq d\right).
$$ 
Propositions~\ref{lsdistHconsist}-\ref{lsdistASconsist} show that any
accumulation point of the limiting distribution of the sequence
$(\hat\sigma\xi_{i,n}\eta_{i,n})^{-1}(\tilde\theta_i - \theta_i)$ with
respect to weak convergence is a measure concentrated on $[-1,1]$,
which proves the result for $d>1$. If $d<1$, the same propositions
show that we can always find a sequence $\theta^{(n)}$ such that the
distribution of $(\hat\sigma\xi_{i,n}\eta_{i,n})^{-1}(\tilde\theta_i -
\theta_{i,n})$ is concentrated on one of the endpoints of the interval
$[-1,1]$ implying
$$
P_{n,\theta,\sigma}
\left(-d \leq (\hat\sigma\xi_{i,n}\eta_{i,n})^{-1}(\tilde\theta_i-\theta_{i,n})
\leq d\right) \longrightarrow 0,
$$ 
proving the claim for $d < 1$ also. 
\end{proof}

\end{appendix}


\newpage

\bibliographystyle{ecrev}

\bibliography{journalsFULL,economet,stat}

\end{document}